\documentclass[12pt]{article}
\def\date{15 Sep 2014}
\usepackage{amsmath, amssymb, amsfonts, amsthm, color}

\newtheorem{proposition}{Proposition}[section]
\newtheorem{lemma}[proposition]{Lemma}
\newtheorem{corollary}[proposition]{Corollary}
\newtheorem{theorem}[proposition]{Theorem}
\newtheorem{claim}[proposition]{Claim}

\newtheorem{question}[proposition]{Question}

\newtheorem{lem}[proposition]{Lemma}
\newtheorem{cor}[proposition]{Corollary}
\newtheorem{thm}[proposition]{Theorem}

{
\theoremstyle{definition}
\newtheorem{definition}[proposition]{Definition}

}

\newcommand{\mcal}{\mathcal}

\newcommand{\C}{\mcal{C}}

\renewcommand{\bar}{\overline}

\begin{document}
\font\smallrm=cmr8




\baselineskip=12pt
\phantom{a}\vskip .25in
\centerline{{\bf Characterizing $4$-critical graphs with Ore-degree at most Seven}}
\vskip.4in
\centerline{{\bf Luke Postle}
\footnote{\texttt{lpostle@uwaterloo.ca. Partially supported by NSERC under Discovery Grant No. 2014-06162.}}} 
\smallskip
\centerline{Department of Combinatorics and Optimization}
\smallskip
\centerline{University of Waterloo}
\smallskip
\centerline{Waterloo, Ontario, Canada}

\vskip 1in \centerline{\bf ABSTRACT}
\bigskip

{
\parshape=1.0truein 5.5truein
\noindent

Dirac introduced the notion of a \emph{$k$-critical} graph, a graph that is not $(k-1)$-colorable but whose every proper subgraph is $(k-1)$-colorable. Brook's Theorem states that every graph with maximum degree $k$ is $k$-colorable unless it contains a subgraph isomorphic to $K_{k+1}$ or an odd cycle (for $k=2$). Equivalently, for all $k\ge 4$, the only $k$-critical graph of maximum degree $k-1$ is $K_k$. A natural generalization of Brook's theorem is to consider the \emph{Ore-degree} of a graph, which is the maximum of $d(u)+d(v)$ over all $uv\in E(G)$. Kierstead and Kostochka proved that for all $k\ge 6$ the only $k$-critical graph with Ore-degree at most $2k-1$ is $K_k$. Kostochka, Rabern and Steibitz proved that the only $5$-critical graphs with Ore-degree at most $9$ are $K_5$ and a graph they called $O_5$. 

A different generalization of Brook's theorem, motivated by Hajos' construction, is Gallai's conjectured bound on the minimum density of a $k$-critical graph. Recently, Kostochka and Yancey proved Gallai's conjecture. Their proof for $k\ge 5$ implies the above results on Ore-degree. However, the case for $k=4$ remains open, which is the subject of this paper.

Kostochka and Yancey's short but beautiful proof for the case $k=4$ says that if $G$ is a $4$-critical graph, then $|E(G)|\ge (5|V(G)|-2)/3$. We prove the following bound which is better when there exists a large independent set of degree three vertices: if $G$ is a $4$-critical graph $G$, then $|E(G)|\ge 1.6 |V(G)| + .2 \alpha(D_3(G)) - .6$, where $D_3(G)$ is the graph induced by the degree three vertices of $G$. As a corollary, we characterize the $4$-critical graphs with Ore-degree at most seven as precisely the graphs of Ore-degree seven in the family of graphs obtained from $K_4$ and Ore compositions.
}

\vfill \baselineskip 11pt \noindent \date.
\vfil\eject
\baselineskip 18pt

\section{Introduction}
All graphs considered in this paper are simple and finite. Graph coloring is an important area of study in graph theory.

\subsection{Ore-Degree}

We know that the chromatic number of a graph is at most the the maximum degree plus one. It is natural to ask for what graphs does equality hold. Since chromatic number and maximum degree are both monotone properties, it suffices to consider the minimal non-colorable graphs.

We say a graph is \emph{$k$-critical} if $G$ is not $(k-1)$-colorable but every proper subgraph is.

Brooks~\cite{Brooks} proved the following theorem characterizing the $k$-critical graphs with maximum degree $k-1$:

\begin{theorem}[Brooks]
For all $k\ge 4$, the only $k$-critical graph with maximum degree $k-1$ is $K_k$. In addition, the only $3$-critical graphs with maximum degree two are odd cycles.
\end{theorem}

An interesting question is to ask if whether Brooks' theorem can be improved. One manner in which to ask this question is to consider the maximum degree of edges instead of vertices. To that end, we define the \emph{Ore-degree} of a graph as the maximum of $d(u)+d(v)$ for every edge $uv\in E(G)$. Brook's Theorem says that for all $k\ge 4$ if $G$ is a $k$-critical graph with Ore-degree at most $2k-2$, then $G$ is isomorphic to $K_k$. Kostochka and Kierstead~\cite{KK} extended Brooks' theorem to graphs with Ore-degree at most $2k-1$ for all $k\ge 6$:

\begin{thm}[Kostochka and Kierstead]\label{OreD5}
For all $k\ge 6$, the only $k$-critical graph with Ore-degree at most $2k-1$ is $K_k$.
\end{thm}

This is best possible since $C_5$ join $K_{k-3}$ is $k$-critical but has Ore-degree $2k$. Note that the only $3$-critical graphs are odd cycles. The remaining cases to consider are $k=4$ and $k=5$. Settling the case $k=5$, Kostocha, Rabern and Steibitz~\cite{KRS} characterized the $5$-critical graphs with Ore-degree at most $9$. There are only two such graphs, $K_5$ and another they called $O_5$. To define $O_5$ let us first define the following construction:

\begin{definition}
An \emph{Ore-composition} of graphs $G_1$ and $G_2$ is a graph obtained as follows:
\begin{enumerate}
\item delete the edge $xy$ from $G_1$;
\item split some vertex $z$ of $G_2$ into two vertices $z_1$ and $z_2$ of positive degree;
\item identify $x$ with $z_1$ and identify $y$ with $z_2$.
\end{enumerate}

We say that $G_1$ is the \emph{edge-side} and $G_2$ the \emph{split-side} of the composition. Furthermore, we say that $xy$ is the \emph{replaced edge} of $G_1$ and that $z$ is the \emph{split vertex} of $G_2$. Finally we say that $G$ is a \emph{$k$-Ore graph} if it can be obtained from copies of $K_k$ and repeated Ore-compositions.
\end{definition}

Note the following easy proposition:

\begin{proposition}\label{OreSep}
If $G$ is a $k$-critical graph and $\{x,y\}$ is a $2$-separation of $G$, then $G$ is the Ore-composition of two $k$-critical graphs where $xy$ is the replaced edge of the edge-side of the composition.
\end{proposition}
\begin{proof}
Let $G=A\cup B$ where $V(A)\cap V(B)=\{x,y\}$. Consider the proper $(k-1)$-colorings of $A$ and $B$. Up to permutation of colors, $x$ and $y$ either receive the same color or different colors. As $G$ is not $(k-1)$-colorable, we may assume without loss of generality that $x$ and $y$ receive the same color in every coloring of $A$ and different colors in every coloring of $B$. Hence $G_1=A+xy$ is not $(k-1)$-colorable. Similarly $G_2=B/xy$ is not $(k-1)$-colorable. Furthermore, $G$ is the Ore-composition of $G_1$ and $G_2$ where $xy$ is the replaced edge of the edge-side $G_1$.

Finally suppose $A$ is not $k$-critical, then there exists a proper subset $A'$ of $A$ which is not $(k-1)$-colorable. If $xy\in A'$, then there exists $(k-1)$-coloring of $(A'-xy)\cup B$ as $G$ is $k$-critical. As $x$ and $y$ receive different colors in every $(k-1)$-coloring of $B$, this induces a $(k-1)$-coloring of $A'$. If $xy\not\in A'$, then $A'$ is a proper subset of $G$ and hence has a $(k-1)$-coloring. Thus $A$ is $k$-critical and similarly $B$ is $k$-critical.
\end{proof}

We may now define $O_5$. Let $G_1,G_2$ be isomorphic to $K_5$, $x,y\in V(G_1)$, and $z\in V(G_2)$. Split $z$ in $G_2$ into two vertices $z_1,z_2$ of degree two. Let $O_5$ denote the Ore-composition of $G_1$ with replaced edge $xy$ and split-side $G_2$ with vertex $z$ split as noted. 

\begin{thm}[Kostochka, Rabern and Steibitz]\label{OreD4}
The $5$-critical graphs with Ore-degree at most $9$ are $K_5$ and $O_5$.
\end{thm}

The remaining case to consider then is $k=4$. The question to answer is: what are the $4$-critical graphs with Ore-degree at most seven? The main purpose of this paper is to answer this question.

Note that for all $k\ge 6$, the only $k$-Ore graph of Ore-degree at most $2k-1$ is $K_k$. Meanwhile the only $5$-Ore graphs of Ore-degree at most $9$ are $K_5$ and $O_5$. Thus it is reasonable to conjecture that the $4$-critical graphs of Ore-degree at most seven are precisely the $4$-Ore graphs of Ore-degree at most seven. Our main theorem asserts that this is indeed the case:

\begin{thm}\label{Main}
A graph $G$ of Ore-degree at most seven is $4$-critical if and only if $G$ is a $4$-Ore graph.
\end{thm} 

Furthermore, we may describe some of the structural properties of $4$-Ore graphs of Ore-degree at most 7 and hence that of $4$-critical graphs of Ore-degree at most 7. First, we need the following lemma.

\begin{lemma}\label{4CritOre7}
If $G$ is a $4$-critical graph of Ore-degree at most 7 and $G$ is the Ore-composition of two $4$-critical graphs $G_1$ and $G_2$, then both $G_1$ and $G_2$ have Ore-degree at most 7.
\end{lemma}
\begin{proof}
Let us prove that they have Ore-degree at most seven. Suppose without loss of generality that $G_1$ is the edge-side with replaced edge $xy$ and $G_2$ is the split-side where $z$ is the split vertex. As $G_1$ is $4$-critical, $x$ and $y$ have degree at least three in $G_1$.  As $G_2$ is $4$-critical, $z$ has degree at least three in $G$.

If $d_{G_1}(x) + d_{G_1}(y)\ge 8$, then $d_G(x) + d_G(y)\ge 9$ and hence one of $x$ and $y$ has degree at least 5 in $G$, contradicting that $G$ has Ore-degree at most 7. So $d_{G_1}(x)+d_{G_1}(y)\le 7$. Note that $d_{G_1}(x)\le d_G(x)$ and $d_{G_1}(y) \le d_G(y)$. This implies that the sum of the degree of every edge in $G_1$ other than $xy$ is at most the sum of the degrees in $G$. Hence $G_1$ has Ore-degree at most 7.

Similarly we claim that $d_{G_2}(z) \le d_G(x), d_G(y)$. If $d_{G_2}(z)=3$, then this follows as $G$ is $4$-critical. So we may suppose that $d_{G_2}(z)=4$. But then $d_G(x)+d_G(y) = d_{G_1}(x)+d_{G_2}(y) - 2 + d_{G_2}(z) = d_{G_1}(x) + d_{G_2}(y) + 2$. As $x$ and $y$ have degree at least 3 in $G_1$ but degree at most 4 in $G$, we find that $d_{G_1}(x)=d_{G_2}(y)=3$ and $d_G(x)=d_G(y)=4$. This proves the claim. It follows then that the sum of the degree of every edge in $G_2$ are at most the sum of the degrees in $G$ and hence $G_2$ has Ore-degree at most seven.
\end{proof}

Kostochka and Yancey proved that if a $4$-Ore graph is the Ore-composition of two graphs, then both of those graphs are $4$-Ore. Using this fact and Lemma~\ref{4CritOre7}, we may prove the following structural characterization of $4$-Ore graphs of Ore-degree at most 7. First we need the following definitions. We say $K_4-e$ subgraph of a graph $G$ whose vertices of degree three are also degree three in $G$ is a \emph{diamond}. We call the vertices of degree two in the $K_4-e$ the \emph{ends} of the diamond and the other vertices we call \emph{internal}. Let $H_7$ denote the unique Ore-composition of two $K_4$s.

\begin{lem}
If $G$ is a $4$-Ore graph of Ore-degree at most 7, then either
\begin{enumerate}
\item $G=K_4$ or $H_7$, or
\item $G$ has a diamond where both ends have degree 4, or
\item $G$ has a diamond where one end has degree 4 and the other end has degree 3 and a neighbor of degree 4.
\end{enumerate}
\end{lem}
\begin{proof}
We proceed by induction on the number of vertices. We may assume that (1) does not hold. As $G\ne K_4$, $G$ is the Ore-composition of two graphs, an edge-side $G_1$ with replaced edge $xy$ and a split-side $G_2$ with split vertex $z$. By Kostochka and Yancey, $G_1$ and $G_2$ are $4$-Ore and by Lemma~\ref{4CritOre7}, $G_1$ and $G_2$ have Ore-degree at most 7. Hence, at least one of (1)-(3) holds for each of $G_1$ and $G_2$.

Note that if $d_{G_2}(z)=3$, then either $\{d_{G_1}(x),d_{G_1}(y)\}=\{3,3\}$ and $\{d_G(x),d_G(y)\}=\{3,4\}$, or, $d_{G_1}(x),d_{G_1}(y)\}=\{3,4\}$ and $\{d_G(x),d_G(y)\}=\{4,4\}$. Similarly if $d_{G_2}(z)=4$, then $\{d_{G_1}(x),d_{G_2}(y)\}=\{3,3\}$ and $\{d_G(x),d_G(y)\}=\{4,4\}$. The following is a useful claim:

\begin{claim}\label{3Then3}
If $z$ is degree 3 in $G_2$, then at least one of $x$ or $y$ is degree $3$ in $G$.
\end{claim}
\begin{proof}
Suppose not. Then as both are degree $4$ in $G$, it follows that $z$ is not adjacent to a degree $4$ in $G_2$. By induction, at least one of (1)-(3) holds for $G_2$. 

Suppose (1) holds for $G_2$. If $G_2=K_4$, then one of $x$ or $y$ is the end of a diamond in $G$ resulting from splitting $G_2$ while the other is adjacent to the other end of that diamond which is degree 3. Hence (3) holds for $G$. If $G_2=H_7$, then $z$ is a vertex of degree three not adjacent to the unique vertex of degree 4 in $H_7$. But then there is a diamond in $G_2$ not including $z$. That diamond is also a diamond of $G$. Moreover, its end of degree 3 is adjacent to $z$ in $G_2$. Hence that end is now adjacent to a degree 4 in $G$. So (3) holds for $G$. So we may suppose that one of (2) or (3) holds for $G_2$. But then, as $z$ is degree 3 in $G_2$ and not adjacent to a degree 4 in $G_2$, it follows that (2) or (3) holds for $G$ as well.
\end{proof}

By induction, at least one of (1)-(3) holds for $G_1$. First suppose (2) or (3) holds for $G_1$. Then the same holds for $G$ unless $xy$ is an edge of the diamond or if (3) holds is the edge from the end of degree 3 to its neighbor of degree 4. Now if $x$ and $y$ both have degree 3 in $G_1$, then there will be two adjacent degree 4s in $G$ since at least one of $x$ or $y$ will be degree 4 in $G$. So we may suppose that one of $x$ or $y$ is degree three in $G_1$ and the other is degree $4$ in $G_1$. Hence both $x$ and $y$ are degree 4 in $G$ and $d_{G_2}(z)=3$, contradicting Claim~\ref{3Then3}.

Finally suppose that (1) holds for $G_1$. Suppose $G_1=H_7$. If $x$ and $y$ are both degree 3 in $G$, then at least one of $x$ and $y$, say $x$, is degree 4 in $G$. But then $x$ is not adjacent to the unique vertex of degree 4, call it $v$, in $H_7$. If $y$ is the other vertex not adjacent to $v$ in $H_7$, then $x$ and $v$ are the ends of a diamond in $G$ and so (2) holds for $G$. If not, then $x$ is adjacent to the end of a diamond in $G$ whose other end is $v$. So (3) holds for $G$. So we may assume without loss of generality that $x=v$, the unique vertex of degree 4 in $H_7$. But then $d_{G_2}(z)=3$ and $y$ has degree $4$ in $G$, contradicting Claim~\ref{3Then3}.

So we may suppose that $G_1=K_4$. Thus $x$ and $y$ are degree 3 in $G_1$ and are the ends of a diamond in $G$. If both $x$ and $y$ are degree $4$ in $G$, the (2) holds. So we may assume without loss of generality that $y$ is degree 4 in $G$ and $x$ is degree 3 in $G$. Thus $z$ is degree 3 in $G_2$. Moreover, the neighbor of $x$ in $G$ not in the diamond is degree 4, then (3) holds for $G$. So we may assume that $z$ is not adjacent to a degree $4$ in $G_2$. Thus if (2) or (3) holds for $G_2$, then the same holds for $G$. So we may assume that (1) holds for $G_2$. If $G_2=K_4$, then $G=H_7$ and (1) holds for $G$. 

So we may assume that $G_2=H_7$. Hence $z$ is not adjacent to the unique vertex of degree $4$ in $H_7$. There are two possible splits of $z$, but in either case (2) or (3) holds in $G$.
\end{proof}

Since $K_4$ and $H_7$ also have diamonds, we obtain the following corollary which confirms a conjecture of Kierstead and Rabern~\cite{KR}:

\begin{cor}\label{HasDiamond}
Every $4$-critical graph $G$ with Ore-degree at most seven has a diamond.
\end{cor}

Corollary~\ref{HasDiamond} says that $4$-critical graphs with Ore-degree at most seven are obtained from a restricted class of Ore-compositions, namely where the edge-side $G_1$ is always isomorphic to $K_4$. Moreover, since contracting a diamond in a $4$-critical graph of Ore-degree at most seven yields another $4$-critical graph of Ore-degree at most 7, Corollary~\ref{HasDiamond} implies that every such graph may be reduced to $K_4$ by a sequence of diamond contractions. Indeed, using Corollary~\ref{HasDiamond}, it is straightforward to characterize the explicit structure of $4$-critical graphs of Ore-degree at most seven; however, we omit its overly technical statement here.

\subsection{Ore's Conjecture}

It is natural to ask what the minimum density in a $k$-critical graph is. As every vertex must have degree at least $k-1$, $\frac{k-1}{2}$ is a trivial lower bound. Inspired by Hajos' construction, Gallai conjectured that the minimum density is in fact $\frac{k}{2} - \frac{2}{k-1}$. Kostochka and Yancey~\cite{KY1} recently resolved Gallai's conjecture:

\begin{thm}[Kostochka and Yancey]\label{OreK}
For all $k\ge 4$, if $G$ is a $k$-critical graph, then

$$|E(G)|\ge (\frac{k}{2}-\frac{1}{k-1})|V(G)| - \frac{k(k-3)}{2(k-1)}$$
\end{thm}

Their proof is quite innovative. Moreover, Theorem~\ref{OreK} provides short proofs of Theorem~\ref{OreD5} when $k\ge 6$ and of Theorem~\ref{OreD4} when $k=5$. However, Theorem~\ref{OreK} with $k=4$ does not imply a characterization of the $4$-critical graphs with Ore-degree at most seven. Nevertheless, this will be the starting point of our proof. Of special interest to us then is a shorter version of their proof for $k=4$ (see~\cite{KY2}):

\begin{theorem}[Kostochka and Yancey]\label{Ore4}
If $G$ is a $4$-critical graph on $n$ vertices, then $$|E(G)|\ge \frac{5n-2}{3}$$
\end{theorem}

We prove a similar but more complicated theorem which yields a short proof of Theorem~\ref{Main}. If $G$ is a graph, let $D_3(G)$ denote the graph induced by vertices of degree at most three in $G$.

\begin{theorem}\label{Ore4Better}
If $G$ is a $4$-critical graph on $n$ vertices, then 

$$|E(G)| \ge 1.6n + .2\alpha(D_3(G))-.6$$
\end{theorem}

Indeed we prove the following stronger theorem which shows that equality holds only if $G$ is a $4$-Ore graph:

\begin{theorem}\label{Ore4Best}
If $G$ is a $4$-critical graph on $n$ vertices, then $|E(G)| \ge 1.6n + .2\alpha(D_3(G))-.6$ if $G$ is a $4$-Ore graph and $|E(G)| \ge 1.6n + .2\alpha(D_3(G))-.4$ otherwise.
\end{theorem}

The proof itself is also of interest for a few reasons. One, we modify the potential developed by Kostochka and Yancey to include the indepence number of $D_3(G)$; this raises the question whether similar improvements are possible for general $k$. Two, we introduce some new theory for $k$-Ore graphs which is also useful for other results in this area. Three, we use an iterated discharging rule, that is, a discharging which make take an arbitrary number of steps. This is one of the only examples of this more complicated discharging. Furthermore, we use the iterated discharging not to move charge along an arbitrarily long path but rather to force charge outward from an arbitrarily nested struture.

\subsection{Organization of the Paper}

The paper is organized as follows. The proof of Theorem~\ref{Ore4Best} comprises Sections 2-6. The proof is similar to that of Kostochka and Yancey for $4$-critical graphs and yet more complicated and intricate.

In Section 2, we modify the potential of Kostochka and Yancey to incorporate indepenent sets in $D_3(G)$. We then prove similar lemmas about this new potential when used in the key reduction from Kostochka and Yancey's proof. Finally, we develop a theory of 'collapsible' subsets which happen to have the least possible potential in a minimum counterexample. 

In Section 3, we briefly develop some straightforward bounds on the potential of proper subsets of a minimum counterexample. 
In Section 4, we improve upon these straightforward bounds by excluding an 'identifiable pair'. This has many corollaries, in particular, that every component of $D_3(G)$ is acyclic. 
In Section 5, we continue to follow Kostochka and Yancey's proof by utilizing a reduction that identifies the neighbors of degree three vertices. Unlike Kostochka and Yancey though, who showed the components of $D_3(G)$ for their minimal counterexample are either a vertex or edge, we show the components of $D_3(G)$ are small, having size at most 10.
In Section 6, we use discharging to finish the proof of Theorem~\ref{Ore4Best}. Indeed we will need an iterated discharging rule to send charge out from the arbitrarily nested structures that arise in Section 5, which we call 'gadgets', and toward the components of $D_3(G)$.

In Section 7, we use Theorem~\ref{Ore4Best} to provide a short proof of Theorem~\ref{Main}. In Section 8, we discuss a few open questions.

\section{Potential, Critical Extensions and Collapsible Subsets}

We now update the potential notion developed by Kostochka and Yancey for $4$-critical graphs to work with independent sets of vertices of degree three.

\begin{definition}
Let $G$ be a graph. We let $D_3(G)$ denote the graph induced by the vertices of degree at most three in $G$. We define the \emph{potential} of subset $R$ of $V(G)$, denoted by $p(R)$, as follows:

$$p(R) = 4.8|R| - 3|E(G[R])| + .6\alpha(G[D_3(G)\cap R])$$

We define the potential $p(G)$ of $G$ to be $p(V(G))$. Similarly we define $P(G)=\min_{H\subseteq G} p(H)$.
\end{definition}

Note that $p(K_1)=5.4$, $p(K_2)=7.2$, $p(K_3)=6$ and $p(K_4)=1.8$. Let $H_7$ be the unique graph that is the Ore composition of two copies of $K_4$. Note $p(H_7)=1.8$.

\subsection{Potential of $4$-Ore Graphs}

We now characterize the potential of $4$-Ore graphs. Note the following:

\begin{proposition}\label{OreDeg4}
If a $4$-critical graph $G$ is the Ore-composition of edge-side $G_1$ and split side $G_2$ with replaced edge $xy$, then at least one of $x$ or $y$ has degree at least four in $G$.
\end{proposition}
\begin{proof}
Since $G_1$ is $4$-critical, $x$ and $y$ have degree at least three in $G_1$. Since $G_2$ is $4$-critical, its split vertex $z$ has degree at least three in $G_2$. Hence at least one of $x$ or $y$ are adjacent to two vertices in $V(G_2)\setminus z$ and so has degree at least four in $G$.
\end{proof}

\begin{lemma}\label{PotentialOreComp}
If $G$ is an Ore composition of $G_1$ and $G_2$, then $p(G)\le p(G_1)+p(G_2) - 1.8$.
\end{lemma}
\begin{proof}
Let $G_1$ be the edge-side of the composition with replaced edge $xy$. Let $G_2$ be the split side of the composition with split vertex $z$. Note that $|E(G)|=|E(G_1)|+|E(G_2)|-1$ and $|V(G)|=|V(G_1)|+|V(G_2)|-1$. Let $I$ be a maximum independent set of $D_3(G)$. Since at most one of $x,y$ is degree three in $G$ by Proposition~\ref{OreDeg4}, it follows that $I\cap V(G_1)$ is an independent set in $D_3(G_1)$. Furthermore, $I\setminus V(G_1)$ is an independent set in $D_3(G_2)$. Thus $\alpha(D_3(G))\le \alpha(D_3(G_1)) + \alpha(D_3(G_2))$. Combining these calculations, we find that $p(G)\le p(G_1)+p(G_2)-4.8+3 = p(G_1)+p(G_2)-1.8$.
\end{proof}

\begin{theorem}\label{OreComp}
If $G$ is $4$-Ore, then $p(G)\le 1.8$.
\end{theorem}
\begin{proof}
We proceed by induction on the number of vertices in $G$. Note that $p(K_4)=1.8$. So we may suppose that $G$ is not isomorphic to $K_4$. Thus $G$ is the Ore composition of two smaller $4$-Ore graphs $G_1$ and $G_2$. By induction, $p(G_1),p(G_2)\le 1.8$. By Lemma~\ref{PotentialOreComp}, $p(G)\le p(G_1)+p(G_2)-1.8 \le 1.8+1.8-1.8 = 1.8$ as desired.
\end{proof}

Later on, we will need the following structural lemma about $4$-Ore graphs whose potential is maximum:

\begin{lemma}\label	{Ind}
If $G$ be a $4$-Ore graph such that $p(G)=1.8$ and $v\in V(G)$, then every maximum independent set of $D_3(G)$ intersects $\bar{N(v)}$.
\end{lemma}
\begin{proof}
We proceed by induction on vertices. Let $v\in V(G)$ and let $I$ be a maximum independent set in $D_3(G)$. If $G=K_4$, then $I\cap \bar{N(v)} \ne \emptyset$ as desired since $\bar{N(v)}=V(G)$ for all $v\in V(K_4)$. 

So we may assume that $G$ is the Ore-composition of two $4$-Ore graphs. Let $G_1$ be the edge-side of this composition with replaced edge $xy$ and let $G_2$ be the split side with split vertex $z$. As $p(G)=1.8$, it follows from Lemma~\ref{PotentialOreComp} that $p(G_1)=p(G_2)=1.8$. Since equality holds throughout it follows that $\alpha(D_3(G))=\alpha(D_3(G_1))+\alpha(D_3(G_2))$. This in turn implies that $I_1=I\cap V(G_1)$ is a maximum independent set in $D_3(G_1)$ and $I_2=I\setminus V(G_1)$ is a maximum independent set in $D_3(G_2)$.

Now we consider two cases. First suppose $v\in V(G_2)\setminus \{z\}$. By induction $I_2$ intersects $\bar{N_{G_2}(v)}$ and hence $I$ intersects $\bar{N_G(v)}$ as desired. So we may assume that $v\in V(G_1)$. By induction, $I_1$ intersects $\bar{N_{G_1}(v)}$. This would imply that $I$ intersects $\bar{N_G(v)}$ as desired unless $v\in \{x,y\}$ and the lone vertex of $I$ in $N_{G_1}(v)$ is the other vertex in $\{x,y\}$. 

Without loss of generality suppose that $v=x$ and that $y\in I$. On the other hand, by induction $I_2$ intersects $\bar{N_{G_2}(z)}$. Let $w\in I_2\cap \bar{N_{G_2}(z)}$. Note that $w\ne z$ since $I_2$ is a subset of $V(G_2)\setminus z$. Thus $w$ is adjacent to one of $x$ and $y$. Since $I$ is an independent set, $w$ is not adjacent to $y$. Thus $w$ is adjacent to $x$. So $I$ intersects $\bar{N_G(v)}$ as desired.    
\end{proof}

\subsection{Critical Extensions}

We will also need the following identification of Kostochka and Yancey to drive the induction.

\begin{definition}
If $R\subsetneq V(G)$ with $|R|\ge 4$, and $\phi$ is a $3$-coloring of $G[R]$, we define the \emph{$\phi$-identification of $R$ in $G$}, denoted $G_{\phi}(R)$, to be the graph obtained by identifying the vertices colored $i$ in $R$ to a vertex $x_i$ for each $i\in\{1,2,3\}$, adding the edges $x_1x_2,x_1x_3,x_2x_3$ and then deleting parallel edges. 
\end{definition}

\begin{proposition}\label{Phi}
If $G$ is $4$-critical, $R\subsetneq V(G)$ with $|R|\ge 4$, and $\phi$ is a $3$-coloring of $G[R]$, then $\chi(G_{\phi}(R))\ge 4$.
\end{proposition}

\begin{definition}
Let $G$ be a $4$-critical graph, $R\subsetneq V(G)$ with $|R|\ge 4$ and $\phi$ be a $3$-coloring of $G[R]$. Now let $W$ be a $4$-critical subgraph of $G_{\phi}(R)$ and $T$ be the triangle corresponding to $R$ in $G$. Then we say that $R' = (W-T)\cup R$ is a \emph{critical extension} of $R$ with \emph{extender} $W$. We refer to $W\cap T$ as the \emph{core} of the extension.

If a vertex in $W-T$ has more neighbors in $R$ then in $X$ or there exists an edge in $G[V(W-T)]$ that is not in $W-T$, then we say that the extension is \emph{incomplete}. Otherwise, we say the extension is \emph{complete}. If $R'=V(G)$, then we say the extension is \emph{spanning}. A \emph{total} extension is an extension that is both complete and spanning.
\end{definition}

Note that - as $G$ is critical - every critical extension has a non-empty core. Here is a useful lemma about the potential of a critical extension.

\begin{lemma}\label{ExtForm}
Let $G$ be a $4$-critical graph, $R\subsetneq V(G)$ with $|R|\ge 4$ and $R'$ be a critical extension of $R$ with extender $W$. Then 

$$p(R') \le p(R) + p(W) - 4.8/6.6/5.4$$

if the core has size $1/2/3$, respectively. Furthermore, if the extension is incomplete, then $p(R')\le p(R) + p(W) -7.8$.
\end{lemma}
\begin{proof}
Note that $|V(R)|+|V(W)| - |V(R')| = |W\cap X|$ and $|E(R)|+|E(W)|-|E(R')|\le |E(G_{\phi}(R)[W\cap X])|$. Finally, $\alpha(D_3(G[R]))+\alpha(D_3(G[W])\ge\alpha(D_3(G[R']))$ as a maximum independent set in $D_3(G[R'])$ corresponds to the disjoint union of independent sets in $D_3(G[R])$ and $D_3(G[W])$.

Thus $p(R)+p(W)-p(R') \ge 4.8(|W\cap X|) - 3\binom{|W\cap X|}{2}$. If $|W\cap X|=1$, then this is $4.8(1)-3(0) = 4.8$. If $|W\cap X|=2$, then this is $4.8(2)-3(1)=6.6$. If $|W\cap X|=3$, this is $4.8(3)-3(3)=5.4$. Finally if the extension is incomplete, then $|E(R)|+|E(W)|-|E(R')|\le |E(G_{\phi}(R)[W\cap X])| + 1$ and hence an additional three is added.
\end{proof}

\subsection{Collapsible Sets}

We will now characterize the subsets whose critical extensions have core size exactly one.

\begin{definition}
Let $G$ be a graph and $R\subsetneq V(G)$ with $|R|\ge 2$. The \emph{boundary} of $R$ is the set of vertices in $R$ with neighbors in $G\setminus R$. If $G$ is $4$-critical, we say $R$ is \emph{collapsible} if in every $3$-coloring of $G[R]$ all vertices in the boundary of $R$ receive the same color. If $R$ is collapsible, then we define the \emph{critical complement} of $R$ to be the graph obtained by identifying the boundary of $R$ to one vertex $v$ and deleting the rest of $R$. We call $v$ the \emph{collapsed vertex} of $W$.

Note then that the boundary of $R$ is an independent set and for any $u,v$ in the boundary of $R$, $G[R]+uv$ contains a $4$-critical subgraph. We say a collapsible subset is \emph{tight} if for any $u,v$ in the boundary of $R$, $G[R]+uv$ is $4$-critical.
\end{definition}

\begin{proposition}
If $R$ is a collapsible subset of a $4$-critical graph $G$, then the critical complement $W$ of $R$ is $4$-critical.
\end{proposition}
\begin{proof}
Suppose not. Then either $W$ is $3$-colorable or there exists an edge $e\in E(W)$ such that $W-e$ is not $3$-colorable. 

First suppose that $W$ is $3$-colorable and let $\phi$ be a $3$-coloring of $W$. Let $\phi'$ be a $3$-coloring of $G[R]$. Note that as $R$ is collapsible, every vertex in $\partial R$ receives the same color in $\phi'$. Let $c$ be this color and let $x$ be the collapsed vertex of $W$. We may assume without loss of generality that $\phi(x)=c$ by permuting the colors of $\phi$ if necessary. But then $\phi\cup \phi'$ is a $3$-coloring of $G$, a contradiction. 

So we may suppose that there exists $e\in E(W)$ such that $W-e$ is not $3$-colorable. But then $e$ corresponds to an edge $e'$ in $G$. As $G$ is $4$-critical, $G-e'$ has a $3$-coloring $\phi$. However, $\phi$ induces a coloring of $R$ and hence every vertex of $\partial R$ receives the same color in $\phi$, call it $c$. Let $\phi'(x)=c$ where $x$ is the collapsed vertex of $W$ and $\phi'(v)=\phi(v)$ for every $v\in V(W)\setminus\{x\}$. Then $\phi'$ is a $3$-coloring of $W-e$, a contradiction.
\end{proof}

This implies that if $R$ is collapsible, then there is only one critical extension of $R$. Indeed, that extension is total, has core size one and the extender is the critical complement. Hence Lemma~\ref{ExtForm} applied to collapsible sets, yields the following characterization of their potential:

\begin{lemma}\label{CollPot}
If $R$ is collapsible subset of a $4$-critical graph $G$ and $W$ is the critical complement of $R$, then $p(R)\ge p(G)-p(W)+4.8$. 
\end{lemma}

Collapsible sets have a total critical extension with core size one. The converse is true if all the extensions have this property: 

\begin{proposition}\label{Coll}
Let $R$ be a proper subset of a $4$-critical graph $G$. Then $R$ is a collapsible subset if and only if every critical extension of $R$ is total and has a core of size one.
\end{proposition}
\begin{proof}
If $R$ is a collapsible subset, then there exists a unique critical extension whose extender is the critical complement. The core of the extension is the special vertex and hence has size one. Moreover, the extension is total. This proves the forward direction.

So let us assume that every critical extension of $R$ is total and has a core of size one. Suppose toward a contradiction that $R$ is not collapsible. Hence there exists a $3$-coloring $\phi$ of $R$ and $u,v\in \partial R$ such that $\phi(u)\ne \phi(v)$. Let $R'$ be an extension of $R$ using $\phi$ with extender $W$. Since $R'$ is total, $R'$ is spanning. Hence all the neighbors of $u,v$ outside of $R$ must be in $R'$. Furthermore, since $R'$ is total, $R'$ is complete. Hence the edges from those neighbors to $u,v$ must be in $W$. But then as $\phi(u)\ne\phi(v)$, the core of the extension must have size at least two, a contradiction.
\end{proof}

As the next proposition asserts, collapsible sets of critical complements yield collapsible sets in the original graph:

\begin{proposition}\label{Coll3}
Let $R$ be a collapsible subset of a $4$-critical graph $G$ and $W$ its critical complement and $v$ its collapsed vertex. If $R'$ is a collapsible set containing $R$, then $R'-R+v$ is collapsible in $W$ and has the same critical complement as $R'$ in $G$.
\end{proposition}
\begin{proof}
Let $\phi$ be $3$-coloring of $A=R'-R+v$. Let $\phi'$ be a $3$-coloring of $R$. We may assume by permuting colors that $\phi(v)=\phi'(z)$ for all $z\in \partial R$. But then $\phi \cup \phi'$ is a $3$-coloring of $R'$. As $R'$ is collapsible, every pair of vertices in $\partial R$ receives the same color. Hence every pair of vertices in $\partial A$ receives the same color in $\phi$. As $\phi$ was arbitrary, it follows that $A$ is collapsible.

Similarly the critical complements are identical since they are both obtained by identifying all the vertices in $R'$ to a single vertex. The latter identification is done in two steps, first by identifying all the vertices in $R$ to a single vertex and then all the vertices in $A$. 
\end{proof}

\subsection{Cocollapsible Sets}

We also would like to characterize the complements of collapsible sets. This motivates the following definitions.

\begin{definition}
Let $(G,R)$ be a rooted graph. If $f: R\rightarrow \{1,\ldots, k\}$, we say $(G,R)$ is \emph{$f-k$-colorable} if there exists a $k$-coloring $\phi$ such that $\phi(v)\ne f(v)$ for all $v\in R$. We say $(G,R)$ is \emph{boundary $k$-colorable} if $(G,R)$ is $f-k$-colorable for all functions $f$ such that there exists $u,v\in R$ such that $f(u)\ne f(v)$. 
\end{definition}

Note that $(K_3,K_3)$ is boundary $3$-colorable.

\begin{definition}
Let $G$ be a $4$-critical graph, $R\subsetneq V(G)$, and the $S$ the boundary of $R$.  We say $R$ is \emph{cocollapsible} if $(G[R], S)$ is boundary $3$-colorable and every vertex in $S$ has exactly one neighbor in $G\setminus R$. We say a cocollapsible subset $R$ is \emph{nontrivial} if $|G\setminus R|>1$.
\end{definition}

Hence a triangle of vertices of degree three is cocollapsible. Note that in a $4$-critical graph $G$, $G\setminus v$ is a trivial cocollapsible subset for every $v\in V(G)$.

\begin{proposition}\label{NontrivialCocoll}
The complement of a nontrivial cocollapsible subset $R$ of a $4$-critical graph $G$ is collapsible.
\end{proposition}
\begin{proof}
Suppose not. Then there exists a $3$-coloring $\phi$ of $G\setminus R$ such that there exist $u,v$ in the boundary of $G\setminus R$ with $\phi(u)\ne \phi(v)$. Now for all $x$ in the boundary of $R$, let $f(x)$ be the color that the neighbor of $x$ in $G\setminus R$ receives in $\phi$. As $R$ is cocollapsible, there exists a coloring $\phi'$ of $R$ with $\phi'(x)\ne f(x)$. But then $\phi\cup \phi'$ is a $3$-coloring of $G$, a contradiction.
\end{proof}

Here are some useful lemmas about the existence of collapsible or cocollapsible sets in $4$-Ore graphs.

\begin{lemma}\label{Cocoll}
Let $G\ne K_4$ be $4$-Ore. For every $v\in V(G)$, $G\setminus v$ contains a nontrivial cocollapsible set.
\end{lemma}
\begin{proof}
We proceed by induction on the number of vertices of $G$. If $G=H_7$, then the lemma follows since for every $v\in V(H_7)$, there exists a triangle of degree three vertices disjoint from $v$ which is a nontrivial cocollapsible subset as desired.

As $G$ is $4$-Ore and $G\ne K_4$, $G$ is the Ore-composition of two $4$-Ore graphs. Let $G_1$ be the edge-side of this composition and $G_2$ the split side with split vertex $z$. If $v\in V(G_1)$, then $G_2\setminus z$ is a trivial cocollapsible set in $G_2$. But then $G_2\setminus z$ is cocollapsible in $G$ and yet nontrivial as desired. So we may suppose that $v\in V(G_2)\setminus z$. 

Suppose $G_2\ne K_4$. Then, by induction, there exists a nontrivial cocollapsible set $R$ in $G_2\setminus v$. Note that $G_2[R]$ is boundary $3$-colorable. If $z\not\in R$, then $G[R]$ is boundary $3$-colorable. Moreover, every vertex in $G[R]$ has at most one neighbor in $G\setminus R$. Thus $R$ is a nontrivial cocollapsible subset of $G$ as desired. So we may suppose that $z\in R$. Let $R'=(R\setminus z)\cup V(G_1)$. Now $G[R']$ is boundary $3$-colorable. Moreover, every vertex in $R'$ has at most one neighbor in $G\setminus R'$. Thus $R'$ is a nontrivial cocollapsible subset of $G$ as desired.

So we may suppose that $G_2=K_4$. Note that since $G\ne H_7$, $G_1\ne K_4$. Let $z_1,z_2$ be the vertices into which $z$ is split. We may suppose without loss of generality that $z_1$ has two neighbors in $V(G_2)$ and $z_2$ has one neighbor in $V(G_2)$. As $G_1\ne K_4$, it follows by induction that there exists a nontrivial cocollapsible $R$ subset of $G_1\setminus z_1$. If $z_2\not\in R$, then $R$ is a nontrivial cocollapsible subset of $G$ as desired. On the other hand if $z_2\in R$, then $R$ is also a nontrivial collapsible subset of $G$, since $z_2$ still has exactly one neighbor outside of $R$, its neighbor in $V(G_2)$ instead of $z_1$.
\end{proof}

\begin{lemma}\label{Cocoll2}
Let $G\ne K_4$ be $4$-Ore. For every triangle $T$ in $G$, $G\setminus T$ contains either a collapsible or a nontrivial cocollapsible set.
\end{lemma}
\begin{proof}
We proceed by induction on the number of vertices of $G$. $G$ is the Ore-composition of two $4$-Ore graphs. Let $G_1$ be the edge-side of this composition and let $G_2$ be the split side with split vertex $z$ split into vertices $z_1,z_2$. But then $T$ is either a subgraph of $G_1$ or $G_2$. Suppose $T$ is a subset of $G_1$. By Lemma~\ref{Cocoll}, there exists a nontrivial cocollapsible subset $R$ of $G_2\setminus z$. But then $R$ is a cocollapsible subset of $G$ as desired.

So we may suppose that $T$ is a subset of $G_2$. If neither $z_1$ nor $z_2$ are in $T$, then $R=V(G_1)$ is a collapsible subset of $G$ as desired. So we may suppose without loss of generality that $z_1\in T$. As $z_2$ is not adjacent to $z_1$, $z_2\not\in T$. If $G_2$ is not isomorphic to $K_4$, then by induction there exists a collapsible or cocollapsible subset of $V(G_2)\setminus T$, which is then a collapsible or cocollapsible set of $G$ as desired.

So we may suppose that $G_2$ is isomorphic to $K_4$. But then $z_2$ has degree one in $G_2(z_1,z_2)$. By Lemma~\ref{Cocoll}, there exists a nontrivial cocollapsible subset $R$ of $G_1\setminus z_1$. If $z_2\not\in R$, then $R$ is certainly cocollapsible in $G$ as desired, while if $z_2\in R$, then $R$ is cocollapsible in $G$ as desired since it follows that $z_2$ has only one neighbor in $G\setminus R$. 
\end{proof}

\subsection{Uncollapsible Vertices}

\begin{definition}
Let $G$ be a $4$-critical graph and $v\in V(G)$. We say $v$ is an \emph{uncollapsible} vertex if there does not exist a collapsible set in $G\setminus v$. We say a graph $G'$ obtained from splitting $v$ into two vertices $v_1,v_2$ is an \emph{uncollapsible} split if there does not exist a collapsible set in $G'$. 
\end{definition}

Note that every vertex of $K_4$ is uncollapsible but yet there is no uncollapsible split of $K_4$. We now characterize uncollapsible vertices in $4$-Ore graphs:

\begin{lemma}\label{Uncoll1}
Let $G$ be a $4$-critical graph such that $G$ is the Ore-composition of edge-side $G_1$ and split-side $G_2$ with split vertex $z$ split into two vertices $z_1,z_2$. If $u\in V(G)$, then $u$ is uncollapsible if and only if $u$ is an uncollapsible vertex of $G_1$, $z$ is uncollapsible in $G_2$, and if $u\in \{z_1,z_2\}$, then $u$ is uncollapsible in the graph obainted from $G_2$ by splitting $z$ into $z_1$ and $z_2$.
\end{lemma}
\begin{proof}
First suppose $u$ in uncollapsible in $G$. That is, there does not exist a collapsible subset disjoint from $u$. Hence $u\in V(G_1)$ since $V(G_1)\setminus V(G_2)$ is collapsible. But then there does not exist a collapsible subset contained in $V(G_2)-z$ and so $z$ is uncollapsible in $G_2$. Finally if $u\in \{z_1,z_2\}$, then certainly $u$ is uncollapsible in the graph obtained from $G_2$ by splitting $z$.

So let us prove the reverse direction. Suppose to a contradiction that $u$ is not uncollapsible and let $R$ be a collapsible subset contained in $G-u$. If $R$ is contained in $G_2$, then $z$ is not uncollapsible in $G_2$, a contradiction. So we may assume that $R$ intersects $G_1\setminus \{z_1,z_2\}$ nontrivially. But then $R\cap V(G_1)$ is a collapsible subset of $G_1$, contradicting that $u$ is uncollapsible in $G_1$.
\end{proof}

In order to bound the size of components of $D_3(G)$ for a minimum counterexample $G$ to Theorem~\ref{Ore4Best}, it will be useful to understand the degree three neighbors of an uncollapsible vertex in a $4$-Ore graph:

\begin{lemma}\label{Uncoll2}
Let $G$ be a $4$-Ore graph and $G(u_1,u_2)$ an uncollapsible split of a vertex $u$ in $G$, then $d(u_1),d(u_2)\ge 2$.
\end{lemma}
\begin{proof}
Suppose not. We may assume without loss of generality that $d(u_1)=1$. But then $G\setminus u_1$ is a collapsible subset, a contradiction.
\end{proof}

Hence there is no uncollapsible split of $K_4$. However, there is an uncollapsible split of $H_7$.

\begin{lemma}\label{Uncoll3}
Let $G$ be a $4$-Ore graph and $G(u_1,u_2)$ an uncollapsible split of a vertex $u$. If $v$ is a neighbor of $u_1$ or $u_2$ with degree three, then either 
\begin{enumerate}
\item $v$ is in a triangle in $G-u$, or,
\item the two neighbors of $v$ not in $\{u_1,u_2\}$ have degree at least four (and there exists an Ore-decomposition of $G$ such that the other two neighbors of $v$ are incident with the replaced edge and $u$ is in the edge-side).
\end{enumerate}
\end{lemma}
\begin{proof}
We proceed by induction on the number of vertices of $G$. By Lemma~\ref{Uncoll2}, $d(u_1),d(u_2)\ge 2$. As $d(u) = d(u_1)+d(u_2)$, we find that $d(u) \ge 4$. The base case to consider then is $H_7$. But every neighbor of the degree four vertex in $H_7$ is in a triangle. 

So we may assume that $G$ is the Ore-composition of two $4$-Ore graphs. Let $G_1$ be the edge-side of this composition and $G_2$ the split side with split vertex $z$ split into two vertices $z_1,z_2$. Further we may assume that $u$ is an uncollapsible vertex in $G_1$, $z$ is an uncollapsible vertex in $G_2$. Indeed, $G_1(u_1,u_2)$ is an uncollapsible split as is $G_2(z_1,z_2)$.

Let $v$ be a degree three neighbor of $u_1$ or $u_2$. We may assume without loss of generality that $v$ is a neighbor of $u_1$. Suppose that $v\in V(G_1)$. By induction on $G_1$,  either $v$ is in a triangle $T=vv_1v_2$ in $G_1-u$ or all its neighbor have degree at least four. We may assume the former as the latter is a desirable outcome. Yet as $v$ has degree three in $G_1$, it follows that $v\not\in \{x,y\}$ as otherwise $v$ would have at least degree two in $G_2$ since $G_2(z_1,z_2)$ is an uncollapsible split and hence $v$ would have degree four in $G$, a contradiction. Now $v$ is a in a triangle in $G$ as desired unless $\{v_1,v_2\}=\{x,y\}$. But then $x,y$ have degree at least four in $G$ as desired.

So we may suppose that $v\in V(G_2)$. Hence $u\in\{x,y\}$. By induction applied to $G_2$, either all the neighbors of $v$ in $G_2$ not in $\{z_1,z_2\}$ are degree  at least four, and hence degree at least four in $G$ as desired, or $v$ is in a triangle in $G_2-z$ and hence a triangle in $G-u$ as desired.
\end{proof}

\section{Properties of a Minimum Counterexample $G_0$ to Theorem~\ref{Ore4Best}}

For the remainder of this paper, let $G_0$ be a counterexample to Theorem~\ref{Ore4Best} with a minimum number of vertices. Hence $G_0$ is not $4$-Ore and $p(G_0) > 1.2$.

\subsection{Subsets with Small Potential}

Next we show that the potential of proper subgraphs of $G_0$ is large. This is where we need critical extensions. First an easy lemma.

\begin{lemma}\label{R0}
For all $R\subsetneq G_0$, $|R|\ge 4$, if $R'$ is a critical extension of $R$, then $p(R)\ge p(R') + 3$.
\end{lemma}
\begin{proof}
Let $R'$ be a $W$-critical extension of $R$. By Lemma~\ref{ExtForm}, $p(R')\le p(R)+ p(W) - 4.8$. As $G_0$ is a minimum counterexample, $p(W)\le 1.8$. Thus $p(R')\le p(R)-3$ as desired. 
\end{proof}

This has the following consequence.

\begin{lemma}\label{R1}
For all $R\subsetneq G_0$, $p(R)\ge p(G_0)+3$.
\end{lemma}
\begin{proof}
Let $R$ be a proper set with minimum potential. Thus $|R|\ge 4$ and $R$ has a critical extension $R'$ as $G$ is $4$-critical. By Lemma~\ref{R0}, $p(R')\le p(R)-3$. As $R$ has minimum potential, it follows that $R'$ is spanning. Thus $p(R)\ge p(R')+3= p(G_0)+3$ as desired.
\end{proof}

Moreover, we can do better if $R$ has an extension that is not total:

\begin{lemma}\label{R2}
If $R\subsetneq G_0$ and $R$ has a critical extension that is not total, then $p(R)\ge p(G_0)+6$.
\end{lemma}
\begin{proof}
First suppose $R$ has a $W$-critical extension $R'$ that is not spanning. By Lemma~\ref{R0}, $p(R)\ge p(R')+3$. By Lemma~\ref{R1}, $p(R')\ge p(G_0)+3$ and hence $p(R)\ge p(G_0)+6$ as desired. So we may suppose that $R$ has a $W$-critical extesnion $R'$ that is spanning but not complete. By Lemma~\ref{ExtForm}, $p(R')\le p(R)+ p(W) - 7.8$. As $G_0$ is a minimum counterexample, $p(W)\le 1.8$. Thus $p(R')\le p(R)-6$ as desired. Rewriting, we find that $p(R)\ge p(R')+6 = p(G_0)+6$ as desired.
\end{proof}

\begin{lemma}\label{R3}
If $R\subsetneq G_0$ and $R$ is not collapsible, then $p(R)\ge p(G_0)+3.6$.
\end{lemma}
\begin{proof}
By Lemma~\ref{R2}, we may assume that every extension of $R$ is total. Since $R$ is not collapsible, there exists by Proposition~\ref{Coll} a critical extension $R'$ of $R$ with extender $W$ whose core has size at least two. By Lemma~\ref{ExtForm}, $p(R')\le p(R)+ p(W) - 5.4$. By the minimality of $G_0$, $p(W)\le 1.8$. Thus $p(R')\le p(R)-3.6$. Rewriting, we find that $p(R)\ge p(R')+3.6=p(G_0)+3.6$ as desired.
\end{proof}

\section{Excluding Identifiable Pairs in $G_0$}

\subsection{Tight Collapsible Sets}

\begin{proposition}\label{TightDeg4}
Let $R$ be a tight collapsible subset of a $4$-critical graph $G$ and $S$ its boundary. If $S$ has size at least three, then every vertex in $S$ has degree at least four in $G$. If $S$ has size two, then at least one vertex has degree at least four in $G$.  
\end{proposition}
\begin{proof}
Suppose $|S|=3$. Let $v\in S$. As $|S|\ge 3$, there exists distinct $u_1,u_2\in S\setminus \{v\}$. As $R$ is tight, $R+u_1u_2$ is $4$-critical. Hence the minimum degree of $R+u_1u_2$ is degree three. This implies that $v$ has degree at least three in $R$. As $v$ is the boundary of $R$, it follows that $v$ has degree at least four in $G$.

So we may suppose that $S=\{u,v\}$. As $R+uv$ is $4$-critical, $u$ and $v$ have degree at least two in $R$. If they both have degree three in $G$, then there exists an edge-cut of size two in $R+uv$, which is impossible. So either $u$ or $v$ has degree at least four in $G$ as desired.
\end{proof}

\begin{definition}
We say $u,v\in V(G)$ is an \emph{identifiable pair} in a proper subset $R$ of $V(G)$ if $u,v \in \partial R$ and $R+uv$ is not $3$-colorable. We say an identifiable pair $(u,v,R)$ is \emph{minimal} if there does not exist $(u',v',R')$ such that $u',v'$ is an identifiable pair in $R'$ and $R'\subsetneq R$.
\end{definition}

Here is a useful lemma:

\begin{lemma}\label{IdPair4Ore}
Let $G$ be a $4$-critical graph. If $(u,v,R)$ is a minimal identifiable pair and $R+uv$ is $4$-Ore, then either there exists a $2$-separation of $G$ or $R$ is not collapsible.
\end{lemma}
\begin{proof}
Suppose not. That is, we may assume that $R$ is collapsible and that there does not exist a $2$-separation of $G$. As $(u,v,R)$ is minimal, it follows that $R$ is a tight collapsible set. Let $K=R+uv$ which by assumption is $4$-Ore. Let $S$ be the boundary of $R$ and let $W$ be the critical complement of $R$. Note that $K\ne K_4$ as otherwise $S$ is a $2$-separation of $G$, a contradiction. Thus $K$ is the Ore-composition of two $4$-Ore graphs. Furthermore, note that $K$ is not the Ore-composition of $4$-Ore graphs $G_1$ and $G_2$ such that $uv\in E(G_2)$ where $G_2$ is the split side of the composition, as otherwise $(x,y,V(G_1))$ contradicts the choice of $(u,v,R)$ since $V(G_1)\subsetneq R$.

Choose $G'$ such that $G'$ is $4$-Ore, $uv\in E(G')$ and $K$ is obatined from $G'$ by repeated Ore-compositions with other $4$-Ore graphs where $G'$ is always contained in the edge-side of the composition, and subject to that, $|V(G')|$ is minimized. Suppose $G'\ne K_4$. Hence $G'$ is the Ore-composition of two $4$-Ore graphs. Let $G_1$ be the edge-side and $G_2$ the split-side of such a composition. As noted above, $uv\not\in E(G_2)$ and hence $uv\in E(G_1)$. But then $G_1$ contradicts the choice of $G'$ since $|V(G_1')| < |V(G_1)|$.

So we may assume that $G'=K_4$.  As there does not exist a $2$-separation of $G$, it follows that $S\setminus V(G') \ne \emptyset$. Let $w\in S\setminus V(G')$. Hence $G$ the Ore-composition of split-side $G_2$ with $v\in V(G_2)$ and edge-side $G'$ whose replaced edge $xy$ is not equal to $uv$. Note then that in every $3$-coloring of $G_2-xy$, $x$ and $y$ receive different colors. We may assume without loss of generality that $u\not\in\{x,y\}$. As $R$ is tight, $K'=R+vw$ is $4$-critical. But then $K'\setminus\{u\}$ has a $3$-coloring $\phi$. Let $a,b$ be the other vertices in $G'\setminus\{u,v\}$. Note that $\phi(a),\phi(b), \phi(v)$ are all distinct. Now we can extend $\phi$ to a $3$-coloring of $K'$ by letting $\phi(u)=\phi(v)$ and possibly recoloring the vertices inside the replaced edges $ua,ub$ if they exist. The latter can be done since $\phi(u)$ is distinct from $\phi(a)$ and $\phi(b)$.
\end{proof}

\begin{lemma}\label{2Sep2}
There does not exist a $2$-separation of $G_0$. That is, the boundary of every proper subset $R$ with $|R|\ge 3$ has size at least three.
\end{lemma}
\begin{proof}
Suppose there exists a $2$-separation $\{x,y\}$ of $G_0$. Then $G$ is an Ore-composition of two graphs $G_1,G_2$ where the replaced edge is $xy$. By Lemma~\ref{PotentialOreComp}, $p(G_0)\le p(G_1)+p(G_2)-1.8$. By the minimality of $G_0$, $p(G_1),p(G_2)\le 1.8$. Furthermore as $G_0$ is not $4$-Ore, at least one of $G_1,G_2$ is not $4$-Ore and thus by the minimality of $G_0$ has potential at most $1.2$. Thus $p(G_0)\le 1.2$, a contradiction.
\end{proof}


\begin{lemma}\label{IdentifiablePair}
There does not exist an identifiable pair of vertices of $G_0$. 
\end{lemma}
\begin{proof}
Let $u,v$ be an identifiable pair in a proper subset $R$ of $V(G_0)$ such that $R+uv$ is not $3$-colorable where we choose $u, v,$ and $R$ such that $|R|$ is minimium. 

Suppose $R$ is collapsible. As $R$ was chosen to have minimum size, it follows that $R$ is tight. Let $K=R+uv$, $S$ be the boundary of $R$ and $W$ be the critical complement of $R$. BY Lemma~\ref{2Sep2}, the boundary of $R$ has size at least three. By Proposition~\ref{TightDeg4}, all vertices in the boundary of $R$ have degree at least four in $G$. In particular, $u$ and $v$ have degree at least four in $G$. Thus every maximum independent set of $D_3(G)$ intersect $R$ is also an independent set in $D_3(K)$. This observation combined with Lemma~\ref{ExtForm} implies that $p(G_0)\le p(K)+3+p(W)-4.8$ where $3$ is added since we delete the edge $uv$ to obtain $R$.  Thus $p(G_0)\le p(K)+p(W)-1.8$. Since $G$ is a minimum counterexample, $p(W),p(K)\le 1.8$. If $K$ is not $4$-Ore, then $p(K)\le 1.2$ and $p(G)\le 1.2+1.8-1.8=1.2$, a contradiction. So we may assume that $K$ is $4$-Ore. By Lemma~\ref{IdPair4Ore}, there exists a $2$-separation of $G$ contradicting Lemma~\ref{2Sep2}.

So we may assume that $R$ is not collapsible. Let $K=R+uv$. Let $R'$ be a critical extension of $R$ with extender $W$ that is either not total or has core size at least two. Note that $p(R)\le p(K) + 3 + .6(\alpha(D_3(G_0[R]))-\alpha(D_3(K))$. Yet $\alpha(D_3(R))\le \alpha(D_3(K))+1$ with equality only if $u$ and $v$ are degree three. Thus $p(R)\le p(K)+3.6$. By Lemma~\ref{ExtForm}, $p(G)\le p(R)+p(W)-5.4 \le p(K)+p(W)-1.8$. By the minimality of $G_0$, $p(K), p(W) \le 1.8$. As $G_0$ is a counterexample, equality holds throughout. Hence it follows that $p(K)=p(W)=1.8$ and so by the minimality of $G_0$, $K$ and $W$ are $4$-Ore. Furthermore, equality implies that $R'$ is a total extension with core size three and also that $\alpha(D_3(R))=\alpha(D_3(K))+1$. The last condition implies that $u$ and $v$ are degree three in $G_0$. Hence $u$ and $v$ are also degree three in $K$. 





Let us further suppose that $K=K_4$. Since $W$ has core size three, it follows that the other two vertices of $K$ are in $W$ and that $W$ is formed by simply identifying $u$ and $v$ to a new vertex $w$. But then $w$ is a degree four vertex in a $4$-Ore graph with potential $1.8$. By Lemma~\ref{Ind}, every maximum independent set of $D_3(W)$ intersects $N(w)$. Thus $\alpha(D_3(G_0))\le \alpha(D_3(W))+1$ (instead of the naive bound $\alpha(D_3(W))+2$ used above). Yet $|V(G_0)|=|V(W)|+1$ and $|E(G)|=|E(W)|+2$. Hence $p(G_0)\le p(W) + 4.8 - 6 +.6 = p(W)-.6 \le 1.2$, a contradiction.

So we may assume that $K\ne K_4$. Thus $K$ is the Ore-composition of two $4$-Ore graphs. Furthermore, $K$ is not the Ore-composition of two graphs $K_1,K_2$ such that $uv\in E(K_2)$ where $K_2$ is the split-side of the composition, as otherwise $(x,y,V(K_1))$ contradicts the choice of $(u,v,R)$ since $V(K_1) \subsetneq R$.

Now choose $K_1$ such that $K_1$ is $4$-Ore, $uv\in E(K_1)$ and $K$ is an Ore-composition of $K_1$ and $K_2$, and subject to that, $|V(K_1)|$ is minimized. It follows from the comment above that $K_1$ is the edge-side of the composition. Suppose $K_1\ne K_4$. But then $K_1$ is the Ore-composition of two $4$-Ore graphs, an edge-side $K_1'$ and a split-side $K_2'$. As noted above, $uv\not\in E(K_2')$ and hence $uv\in E(K_1')$. But then $K_1'$ contradicts the choice of $K_1$ since $|V(K_1')| < |V(K_1)|$. 

So we may assume that $K_1=K_4$. Let $xy$ be the replaced edge of $K_1$, $z$ the split vertex of $K_2$ and $z_1$, $z_2$ the vertices into which $z$ is split in $K$. We claim that $d_{K_2}(z_1), d_{K_2}(z_2)\ge 2$. Suppose not. We may assume without loss of generality that $d_{G_2}(z_1)=1$. Let $w$ be the neighbor of $z_1$ in $G_2$. But now $w$ and $z_2$ form a $2$-cut of $K$. Indeed they yield an Ore-decomposition of $K$ where the split-side is $K_4$ and yet contains the edge $uv$, a contradiction. This proves the claim that $d_{K_2}(z_1), d_{K_2}(z_2)\ge 2$.

It follows that $x,y$ are degree at least four in $K$. Thus $u,v\not\in\{x,y\}$. As $K_1=K_4$, $V(K_1)=\{x,y,u,v\}$. Let $G'$ be obtained from $G_0$ by identifying $u$ and $v$ to a vertex $w$. As $u$ and $v$ must receive the same color in every $3$-coloring of $R$ it follows that $G'$ contains a $4$-critical subgraph $W'$. Moreover, $w\in W'$. Yet $w$ has degree at least three in $W'$. Thus at least one of $x$ or $y$ must be in $W$. Suppose without loss of generality that $x\in W$. 

Suppose $y\not\in W'$. Let $R_0=W'\setminus w \cup \{u,v\}$. Now $p(R_0)\le p(W') + 4.8-3 + 1.2 = p(W)+3$. Now $p(W') \le 1.8$ by the minimality of $G_0$. So $p(R_0)\le 4.8$. By Lemma~\ref{R3}, $p(G_0)\le 1.2$, a contradiction. 

Finally we may suppose that $y\in W'$. In this case, $p(R_0)\le p(W)' + 4.8-6+1.2= p(W')$. It follows from Lemma~\ref{R3}, that $R_0=V(G_0)$. As $G_0$ is a minimum counterexample and $p(G_0) > 1.2$, we find that $p(W')=1.8$ and $W'$ is $4$-Ore. Yet $w$ is a vertex of degree four in a $4$-Ore-graph with potential $1.8$. By Lemma~\ref{Ind}, every maximum independent set of $D_3(W')$ intersects $N(w)$. Thus $\alpha(D_3(G_0))\le \alpha(D_3(W'))+1$. Yet $|V(G_0)|=|V(W')|+1$ and $|E(G_0)|=|E(W')|+2$. Calculating, we find that $p(G_0)\le p(W')+4.8-6+.6=p(W')-.6\le 1.2$, a contradiction.
\end{proof}

Lemma~\ref{IdentifiablePair} has many consequences which we now list.

\begin{corollary}\label{NoColl}
There does not exist a collapsible subset of $G_0$.
\end{corollary}
\begin{proof}
By definition, every collapsible subset contains an identifiable pair contradicting Lemma~\ref{IdentifiablePair}.
\end{proof}

\begin{corollary}\label{NoDiamond}
$G_0$ does not contain a subgraph isomorphic to $K_4-e$.
\end{corollary}
\begin{proof}
Otherwise the two vertices of degree two in that subgraph are an identifiable pair contradicting Lemma~\ref{IdentifiablePair}.
\end{proof}

\begin{corollary}\label{NoCoColl}
There does not exist a nontrivial cocollapsible subset of $G_0$.
\end{corollary}
\begin{proof}
Suppose there exists a nontrivial cocollapsible subset $R$ of $G_0$. As $R$ is nontrivial, $|G_0\setminus R|\ne 1$. By Proposition~\ref{NontrivialCocoll}, $G\setminus R$ is collapsible contradicting Lemma~\ref{NoColl}.
\end{proof}

\begin{corollary}\label{NoCycle3}
$G_0$ does not contain a cycle of vertices of degree three. Hence every component of $D_3(G_0)$ is a tree.
\end{corollary}
\begin{proof}
Suppose to a contradiction that there exists a cycle $C=v_1v_2\ldots v_k$ in $G_0$ such that $d(v_i)=3$ for all $i, 1\le i\le k$. By Gallai's theorem, $k$ is odd and $C$ is induced. Note then that $V(C)$ is a cocollapsible set. By Lemma~\ref{NoCoColl}, $|G\setminus V(C)|=1$. That is, there exists $u\not\in V(C)$ adjacent to all the vertices of $C$. Hence $G_0$ is an odd wheel, $|V(G_0)|=k+1$, $|E(G_0)|=2k$, $\alpha(D_3(G_0))=\frac{k-1}{2}$. So $p(G_0) = 4.8(k+1)-3(2k) + .6\frac{k-1}{2} = 4.5 - .9k$. As $G_0\ne K_4$, $k\ge 5$ and hence $p(G_0)\le 0$, a contradiction.
\end{proof}

\begin{corollary}\label{NoTriangle}
Every triangle of $G_0$ contains at most one vertex of degree three.
\end{corollary}
\begin{proof}
Suppose not. That is there exists a triangle $T=v_1v_2v_3$ such that $d(v_1)=d(v_2)=3$. By Corollary~\ref{NoCycle3}, we may assume that $d(v_3)\ge 4$. Let $N(v_1)\setminus T = \{u_1\}$ and $N(v_2)\setminus T =\{u_2\}$. If $u_1\ne u_2$, then $u_1,u_2$ is an identifiable pair of vertices contradicting Lemma~\ref{IdentifiablePair}. So we may assume that $u_1=u_2$. But then there exists a subgraph of $G_0$ isomorphic to $K_4-e$ contradicting Corollary~\ref{NoDiamond}.
\end{proof}

\begin{corollary}\label{54}
If $R$ is a proper subset of $V(G_0)$, then $p(R)>5.4$ unless $V(G)\setminus R=\{v\}$ where $v$ is a vertex of degree three.
\end{corollary}
\begin{proof}
Suppose not. It is straightforward to check that the statement holds when $|R|\le 3$. So we may assume that $|R|\ge 4$. Let $R'$ be a critical extension of $R$ with extender $W$. As $G_0$ is minimum counterexample, $p(G_0)>1.2$. If $R'$ is not total, then $p(R)\ge p(G_0)+6 \ge 7.2$, a contradiction. Yet we may choose $R'$ with core size at least two since $R$ is not collapsible by Corollary~\ref{NoColl}. Thus $p(G_0)\le p(R)+p(W)-5.4 \le p(W)$. Hence $W$ is $4$-Ore and the extension has core size three. Suppose $W\ne K_4$. By Lemma~\ref{Cocoll2}, there exists a collapsible or cocollapsible subset $R_0$ of $G_0\setminus T$, where $T$ is the core of the extension. Note $T$ is a triangle since the extension has core size three. But $R'$ is a total extesnion of $R$. Thus either $R_0$ or $G_0\setminus R_0$ is collapsible in $G_0$, contradicting Lemma~\ref{NoColl}. So we may assume that $W=K_4$. That is $G_0\setminus R$ consists of one degree three vertex, a contradiction.
\end{proof}

\section{Characterizing Components of $D_3(G_0)$}

We now attempt to characterize the components of $D_3(G_0)$. In particular we show that they are small in size. A useful tool for this goal is the following reduction of Kostochka and Yancey~\cite{KY2}.

\subsection{Degree Three Reductions}

\begin{definition}
Let $v$ be a vertex of degree three in $G_0$ with neighbors $u_1,u_2,u_3$. If $u_1$ is not adjacent to $u_2$ then the graph obtained from $G_0$ by deleting $v$ and identifying $u_1$ and $u_2$ is not $3$-colorable and so contains a $4$-critical subgraph $K$. We say $K$ is a \emph{degree three reduction} of $v$ in the direction of $u_1$ and $u_2$ and denote it by $K(v; u_1,u_2)$. We also say that $K$ is a degree three reduction of $v$ away from $u_3$ and may denote it as $K(v;u_3)$. We say $R=V(K)-u_1u_2 + \{v,u_1,u_2\}$ is the \emph{expansion} of $K$.
\end{definition}


\begin{lemma}\label{DegreeThree}
Let $K$ be a degree three reduction of $G_0$ and let $R$ be the expansion of $K$. If $R'$ is a critical extension of $R$ with extender $W$, then $R'$ is total. In addition if $W$ does not have a core of size one, then either 
\begin{enumerate}
\item $K$ is $4$-Ore, or,
\item $\alpha(D_3(G)) = \alpha(D_3(K))+ \alpha(D_3(W))+2$, or
\item $W$ is $4$-Ore
\end{enumerate}
Furthermore, if $W$ has a core of size two, then (1), (2) and (3) all hold.
\end{lemma}
\begin{proof}
Suppose that $K$ is a reduction of $v$ in the direction of $u$ and $w$. Since $|E(R)|=|E(K)|+2$ and $|V(R)|=|V(K)|+2$, we find that $p(R)= p(K) + 3.6 + .6(\alpha(D_3(R))-\alpha(D_3(K)))$. Yet note that $\alpha(D_3(R))\le \alpha(D_3(K))+2$ since $u$ and $w$ may be degree three in $R$. Thus $p(R) \le p(K)+4.8$. As $G_0$ is a minimum counterexample, $p(K),p(W)\le 1.8$. Moreover, $\alpha(D_3(R'))\le \alpha(D_3(R))+\alpha(D_3(W))$ which is at most $\alpha(D_3(K))+\alpha(D_3(W))+2$.

Suppose $W$ is not complete. By Lemma~\ref{ExtForm}, $p(R')\le p(R)+P(W) - 7.8 \le 1.8 + 4.8 + 1.8 - 7.8 = .6$, a contradiction. Suppose $W$ is not spanning. By Lemma~\ref{ExtForm}, $p(R')\le p(R)+P(W) - 4.8 \le 1.8 + 4.8 + 1.8 - 4.8 = 3.6$. But since $W$ is not spanning, $R'\ne V(G)$. By Lemma~\ref{R1}, $p(G_0)\le p(R')-3 = .6$, a contradiction. Thus we may assume $W$ is total. 

If $W$ has a core of size two, then by Lemma~\ref{ExtForm}, $p(G)\le p(R)+p(W)-6.6 \le 1.8 + 4.8 + 1.8 - 6.6 = 1.8$. If $K$ is not $4$-Ore, then $p(K)\le 1.2$ and hence $p(G_0)\le 1.2$, a contradiction. So (1) holds. Similarly if $W$ is not $4$-Ore, then $p(W)\le 1.2$ and hence $p(G_0)\le 1.2$, a contradiction. So (2) also holds. Finally if $\alpha(D_3(G_0)) \ne \alpha(D_3(K))+\alpha(D_3(K)) + 2$, then $\alpha(D_3(G))\le \alpha(D_3(R))+\alpha(D_3(W)) - 1$; in that case, $p(G_0)\le p(R)+p(W) - 7.2$ and hence $p(G)\le 1.2$, a contradiction. So (3) holds as well.

So we may assume that $W$ has a core of size three. By Lemma~\ref{ExtForm}, $p(G)\le p(R)+p(W)-5.4$. Now suppose that none of (1), (2), or (3) hold. Since (1) does not hold, $K$ is not $4$-Ore. By the minimality of $G_0$, $p(K)\le 1.2$. Since (3) does not hold, $W$ is not $4$-Ore. By the minimality of $G_0$, $p(W)\le 1.2$. Since (2) does not hold, $\alpha(D_3(G_0)) \ne \alpha(D_3(K))+\alpha(D_3(K)) + 2$. Thus $\alpha(D_3(G_0))\le \alpha(D_3(R))+\alpha(D_3(W)) - 1$. This last observation improves the bound from Lemma~\ref{ExtForm} to $p(G_0)\le p(R)+p(W)-6$. Since $p(R)\le p(K)+4.8$, we have that $p(G_0) \le p(K)+p(W) - 1.2$. But since $p(K),p(W)\le 1.2$, it follows that $p(G) \le 1.2+1.2 - 1.2 = 1.2$, a contradiction. 
\end{proof}

\begin{lemma}\label{DegreeThree2}
Let $K(v; u_1,u_2)$ be a degree three reduction of $G_0$ and let $R$ be the expansion of $K$. Then either 
\begin{enumerate}
\item $K$ is $4$-Ore and $K(u_1,u_2)$ is an uncollapsible split, or,
\item $d(u_1)=d(u_2)=3$ and there exists a maximum independent set of $D_3(K)$ that does not intersect $\bar{N(u_1)}\cup \bar{N(u_2)}$,
or,
\item $|V(G)\setminus R|=1$.
\end{enumerate}
\end{lemma}
\begin{proof}
Let $R'$ be a critical extension of $R$ with extender $W$. By Lemma~\ref{NoColl}, we choose $W$ such that it does not have a core of size one. Apply Lemma~\ref{DegreeThree} to $K$, $R$ and $W$. Thus the extension is total and one of \ref{DegreeThree}(1), (2) or (3) holds.

Suppose \ref{DegreeThree}(1) holds. That is, $K$ is $4$-Ore. Thus $K(u_1,u_2)$ is a split of a $4$-Ore graph. As $G_0$ does not contain a collapsible subset by Lemma~\ref{NoColl}, it follows that $K(u_1,u_2)$ is an uncollapsible split. So 1 holds as desired. (**Potential argument needed! Not necc collapsible in $G$!)

Suppose \ref{DegreeThree}(2) holds. Thus $u_1,u_2$ are in every maximum indpendent set of $D_3(G_0)$. So $d(u_1)=d(u_2)=3$ and yet $\alpha(D_3(R)) = \alpha(D_3(K))+2$. Thus the identified vertex is not in some maximum independent of $D_3(K)$ nor its neighbors which form the set $N(u_1)\cup N(u_2)$. Hence there exists a maximum independent set of $D_3(K)$ that does not intersect $N(u_1)\cup N(u_2)$. So 2 holds as desired.

So we may suppose that neither \ref{DegreeThree}(1) or (2) hold. Hence $W$ has a core of size three and \ref{DegreeThree}(3) holds. That is, $W$ is $4$-Ore. If $|V(G_0)\setminus R|=1$, then 3 holds as desired. So we may suppose that $|V(G_0)\setminus R|\ne 1$. Hence $W\ne K_4$. By Lemma~\ref{Cocoll2}, there exists a cocollapsible subset $R_0$ in $W\setminus T$. As the extension is total, every vertex in $R_0$ has at most one neighbor in $V(G_0)\setminus R_0$. Thus $R_0$ is a cocollapsible subset of $G_0$. As $|V(G_0)\setminus R_0|\ne 1$, $G\setminus R_0$ is a collapsible subset of $G_0$, contradicting Lemma~\ref{NoColl}.
\end{proof}

\subsection{Characterizing the Components of Degree Three Vertices}

Note by Corollary~\ref{NoCycle3} that every component of $D_3(G)$ is a tree. Our goal is to show that these components have bounded size. Indeed, we will go further and characterize the possible components exactly. For discharging purposes, we will also need that degree four neighbors of the larger components of degree three are in special structures. To that end, we will need the following notion:

\begin{definition}
If $v$ is a vertex of degree three and $u$ is a neighbor of $v$ with degree at least four, then we say that $u$ is \emph{good for $v$} if $u$ is contained in an uncollapsible split of a $4$-Ore graph not containing $v$.
\end{definition}

We now proceed by analyzing the vertices of degree three in components of $D_3(G)$ of size at least three. We analyze these vertices according to their degree in that component. First we analyze the degree ones (i.e. the leafs):

\begin{lemma}\label{Deg1}
Let $C\in \C(D_3(G_0))$ and $v\in V(C)$. If $d_C(v)=1$ and $|C|\ge 3$, then the neighbors of degree at least four of $v$ are either adjacent or good for $v$. 
\end{lemma}
\begin{proof}
Suppose not. Let $u_1,u_2$ be the neighbors of $v$ with degree at least four. We may suppose that $u_1$ and $u_2$ are not adjacent. Hence we may apply Lemma~\ref{DegreeThree2} to the degree three reduction $K(v;u_1, u_2)$. If \ref{DegreeThree2}(1) holds, then $u_1$ and $u_2$ are good for $v$, a contradiction. Yet \ref{DegreeThree2}(2) does not hold as $d(u_1),d(u_2)\ge 4$. Thus we may suppose that \ref{DegreeThree2}(3) holds. However, $(C\setminus v) \cap R=\emptyset$ and hence $|V(G_0)\setminus R|\ge 2$ as $|C|\ge 3$, a contradiction.
\end{proof}

Next we will analyze the degree twos. But first let us define a notion of smallness that will be useful:

\begin{definition}
Let $C$ be a component of $D_3(G_0)$. If $uv\in E(C)$, then we let $C(v;u)$ denote the component of $C\setminus \{u\}$ containing $v$. We say $C(v;u)$ is \emph{small} if it has at most five vertices and all vertices in it have distance at most two from $v$.
\end{definition}

\begin{lemma}\label{Deg2First}
Let $C\in \C(D_3(G_0))$ and $v\in V(C)$. If $d_C(v)=2$, $|C|\ge 3$ and $u$ is a degree three neighbor of $v$, then either
\begin{enumerate}
\item $u$ is degree one in $C$, or,
\item $C(v;u)$ has size at most 4 and is small and the neighhbor of degree at least four of $v$ is good for $v$.
\end{enumerate}
\end{lemma}
\begin{proof}
Let $w$ be the neighbor of degree at least four of $v$ and $u'$ be the other neighbor of degree three of $v$. By Lemma~\ref{NoTriangle}, $u'$ and $w$ are not adjacent. Hence we may apply Lemma~\ref{DegreeThree2} to the degree three reduction $K_1(v;u', w)$. \ref{DegreeThree2}(2) does not hold as $d(w)\ge 4$. If \ref{DegreeThree2}(3) holds, then $u$ has degree one in $C$ and 1 holds as desired. So we may suppose that \ref{DegreeThree2}(1) holds. Hence $w$ is good for $v$ and by Lemma~\ref{Uncoll3}, there are at most two vertices of degree three from $u'$ away from $v$ and 2 holds as desired.
\end{proof}

\begin{lemma}\label{Deg2}
Let $C\in \C(D_3(G_0))$ and $v\in V(C)$. If $d_C(v)=2$, $|C|\ge 3$ and $w$ is the neighbor of degree at least four of $v$, then either
\begin{enumerate}
\item $|C|\le 3$, or,
\item $|C|\le 7$ and $w$ is good for $v$, or,
\item $v$ has neighbor of degree one in $C$ and $w$ is good for $v$.
\end{enumerate}
\end{lemma}
\begin{proof}
Let $w$ be the neighbor of degree at least four of $v$ and $u_1,u_2$ be the neighbors of degree three of $v$. Apply Lemma~\ref{Deg2First} separately to $u_1$ and $u_2$. If (1) holds for both $u_1$ and $u_2$, then $|C|=3$ since $u_1$ and $u_2$ are both degree one in $C$. Thus 1 holds as desired. If (1) holds for one and (2) for the other, then 3 holds since $w$ is good for $v$ by (2) and $v$ has degree one neighbor in $C$ by (1). Finally if (2) holds for both $u_1$ and $u_2$, then $w$ is good for $v$ and $|C(v;u_1)|,|C(v;u_2)|\le 4$. Thus $|C|\le 7$ and 2 holds as desired.
\end{proof}

\begin{lemma}\label{Deg3First}
Let $C\in \C(D_3(G_0))$ and $v\in V(C)$. Suppose $d_C(v)=3$ and $N(v)=\{u_1,u_2,u_3\}$. Then either 

\begin{enumerate}
\item $C(v;u_3)$ is small, or,
\item for all $i\in \{1,2\}$,  $\alpha(C(u_i;v)) = \alpha(C(u_i;v)-u_i)+1$, or
\item $u_3$ is degree one.
\end{enumerate}
\end{lemma}
\begin{proof}
By Lemma~\ref{NoTriangle}, $u_1$ and $u_2$ are not adjacent. Hence we may apply Lemma~\ref{DegreeThree2} to the degree three reduction $K_1(v;u_1, u_2)$. If \ref{DegreeThree2}(3) holds, then $u_3$ has degree one in $C$ and 3 holds as desired. If \ref{DegreeThree2}(1) holds, then $C(v;u_3)$ is small and (1) holds as desired. So we may suppose \ref{DegreeThree2}(2) holds, and hence there exists a maximum independent set of $C(v;u_3) \setminus\{u_1,u_2,v\}$ not intersecting $N(u_1)\cup N(u_2)$. But then for all $i\in\{1,2\}$ there is a maximum independent set $I_i$ of $C(u_i;v)$ not intersecting $N(u_i)$, which means that $\alpha(C(u_i;v)) = \alpha(C(u_i;v)-u_i)+1$ since $u_i+I_i$ is also an independent set. Thus 2 holds as desired.
\end{proof}

\begin{lemma}\label{Deg3}
Let $C\in \C(D_3(G_0))$ and $v\in V(C)$. Suppose $d_C(v)=3$ and $N(v)=\{u_1,u_2,u_3\}$. Then there exists $i\in\{1,2,3\}$ such that $C(v;u_i)$ is small.
\end{lemma}
\begin{proof}
Suppose not. Then for at least two $i\in\{1,2,3\}$, $|C(u_i;v)|\ge 3$. Suppose without loss of generality that $|C(u_1;v)|, |C(u_3;v)|\ge 3$. Thus $u_1$ and $u_3$ do not have degree one in $C$. Apply Lemma~\ref{Deg3First}. If (1) holds, then $C(v;u_3)$ is small, a contradiction. (3) does not hold since $u_3$ does not have degree one in $C$. Thus 2 holds. Hence $\alpha(C(u_1;v))=\alpha(C(u_1;v)-u_1)+1$.

Now suppose $u_1$ has degree two. Let $w$ be the neighbor of $u_1$ in $C$ distinct from $v$. As $|C(u_1;v)|\ge 3$, $w$ is not degree one in $C$. Hence Lemma~\ref{Deg2First}(2) holds. That is, $C(u_1;w)$ has size at most 4 and is small. But then $C(v;u_1)$ is also small, a contradiction.

So we may assume that $u_1$ has degree three in $C$. Let $w_1,w_2$ be the neighbors of $u_1$ in $C$ distinct from $v$. Apply Lemma~\ref{Deg3First} where $u_1$ plays the role of $v$ in that lemma and $w_2$ plays the role of $u_3$. Suppose (1) holds. That is, $C(u_1;w)$ is small. But then $C(v;u_1)$ is also small, a contradiction. Suppose (2) holds. But then $\alpha(C(w_1;u_1))=\alpha(C(w_1;u_1)-w_1)+1$. This is a contradiction since there exists a maximum indepedent set $I$ in $C(u_1;v)-u_1$ not intersecting $N(u_1)$ which would imply that $\alpha(C(w_1;u_1))=\alpha(C(w_1;u_1)-w_1)$. 

Finally suppose $3$ holds. That is $w_1$ has degree one. But then there does not exist a maximum independent set $I$ in $C(u_1;v)-u_1$ not intersecting $N(u_1)$ since then $I+w_1$ is a larger independent set in $C(u_1;v)-u_1$, a contradiction.
\end{proof}

\begin{lemma}\label{Components}
If $C$ is a component of $D_3(G)$, then $|C|\le 10$. Furthermore, if $|C|\ge 4$, then every neighbor of degree at least four of a vertex $v$ in $C$ is good for $v$ or $d_C(v)=1$ and the two neighbors of degree at least four of $v$ are adjacent.
\end{lemma}
\begin{proof}
We may suppose that $|C|\ge 4$. Let $v\in C$ and $w$ be a neighbor of degree at least four of $v$. Hence $d_C(v)=1$ or $2$. If $d_C(v)=1$, then $w$ is good for $v$ or is adjacent to other neighbor of degree four of $v$ by Lemma~\ref{Deg1}. If $d_C(v)=2$, then $w$ is good for $v$ by Lemma~\ref{Deg2} since $|C|\ge 4$. So it remains to show that $|C|\le 10$.

Suppose that there exists a vertex $v$ of degree three in $C$ with at least two neighbors of degree two in $C$. By Lemma~\ref{Deg3}, there exists a neighbor $u$ of $v$ such that $C(v;u)$ is small and hence $|C(v;u)|\le 5$. If $u$ has degree one in $C$, then $|C|=6$ as desired. Suppose $u$ has degree two in $C$. Apply Lemma~\ref{Deg2} to $u$. (1) does not hold since $|C|\ge 4$. If (3) holds, then $u$ has a neighbor $w$ of degree one in $C$. Since $v$ has degree three in $C$, $v\ne w$ and so $|C|\le 7$ as desired. So we may suppose that (2) holds. But then $|C|\le 7$ as desired.

So we may suppose that $u$ has degree three in $C$. By Lemma~\ref{Deg3}, there exists a neighbor $w$ of $u$ such that $C(u;w)$ is small. Yet $C(u;w)$ is not small for all $w\ne v$ since $v$ has a neighbor $x$ of degree at least two distinct from $u$ and thus $x$ has distance at least three from $u$. Hence $C(u;v)$ is small. But then $|C|\le |C(u;v)|+|C(v;u)| \le 5+5 = 10$ as desired.

So we may suppose that every vertex of degree three in $C$ has two neighbors of degree one in $C$. Next let us suppose there exists a vertex $v$ of degree three in $C$. If $v$ has only neighbors of degree one in $C$, then $|C|=4$ as desired. So we may suppose that $v$ has exactly one neighbor $u$ of degree at least two. Suppose $u$ has degree two in $C$. Apply Lemma~\ref{Deg2} to $u$. (1) does not hold since $|C|\ge 4$. If (3) holds, then $u$ has a neighbor $w$ of degree one in $C$. Since $v$ has degree three in $C$, $v\ne w$ and so $|C|\le 5$ as desired. So we may suppose that (2) holds. But then $|C|\le 7$ as desired.

So we may suppose that $u$ has degree three in $C$. But then $u$ has two neighbors of degree one in $C$. Hence $|C|=6$ as desired.

Finally we may suppose that there exist no vertices of degree three in $C$. That is, $C$ is a path $P=x_1x_2\ldots x_k$. We may suppose that $k\ge 5$ as otherwise $|C|=4$ as desired. But then apply Lemma~\ref{Deg2} to $x_3$. (1) does not hold as $|C|\ge 4$. (3) does not hold since $x_3$ has only neighbors of degree two in $C$. Hence (2) holds and $|C|\le 7$ as desired.
\end{proof}

\section{Proof of the Main Result}

We are now ready to prove Theorem~\ref{Ore4Best}. Here is an equivalent form in terms of potential:

\begin{theorem}
If $G$ is $4$-critical and $G$ is not $4$-Ore, then $p(G)\le 1.2$.
\end{theorem}

We now prove Theorem~\ref{Ore4Best}.

\begin{proof}

Let $G_0$ be a minimum counterexample. As noted we may assume then that $|V(G_0)|\ge 5$.

\subsection{Discharging}

We now utilize Lemma~\ref{Components} as forbidding certain configurations that will allow us to perform discharging.

We define the charge of a vertex $v\in V(G_0)$, denoted by $ch(v)$ as:

$$ch(v)= deg(v)-3.2$$ 

Note that the charge of a vertex of degree three is $-.2$, of degree four is $.8$, and of degree at least five is $1.8$.

We will prove that $\sum_{v\in V(G_0)} ch(v) \ge .4 \alpha(D_3(G_0))$. To do this we apply an initial discharging rule, and then one futher rule
arbitrarily many times.

We use the following intial discharging rule: \\

\emph{Rule 0: If $v$ is a degree three vertex in a triangle $T=vw_1w_2$, then $w_1,w_2$ send $.4$ charge each to $v$.}\\

Let $ch_1(v)$ be the charge of a vertex $v$ after applying Rule 0. For all $i\ge 1$, we apply the following discharging rule where $ch_{i+1}(v)$ is the charge of a vertex $v$ after applying Rule $i$ as long as some charge will be sent by the rule. \\

\emph{ Rule $i$: For every vertex $v$, let $N_i(v) = \{u\in N(v)| d(u)=3, ch_i(v) < .6\}$. If $v$ has degree at least four and $ch_i(v) \ge .4 N_i(v)$, then $v$ sends $.4$ charge to each vertex in $N_i(v)$.} \\

Notice that $ch_i(v)\ge 0$ for all $v$ such that $d(v)\ge 4$. Similarly, for all $v$ with $d(v)=3$ and $ch_i(v) \ge 0$, $ch_{i+1}(v)=ch_i(v)$. Moreover, as some charge is sent there exists a vertex $u$ of degree three such that $ch_i(u) < .6$ and $ch_{i+1}(u) \ge .6$. This implies that eventually we must stop applying this rule since at every application at least one vertex of degree three goes from below $.6$ to at least $.6$ charge and such vertices never go below $.6$ again.

So let $T-1$ be the number of times we apply the above rule. We now apply one final rule:\\

\emph{Rule $T$: If $v$ is a vertex of degree at least four, then $v$ sends $.2$ charge to every vertex of degree three with $ch_T(v) < .6$ that $v$ did not already send charge to under a previous Rule. }

Let $ch_T(v)$ be the final charge of $v$.

\begin{claim}\label{CH4}
For all $v$ of degree at least four, $ch_T(v) \ge 0$.
\end{claim}
\begin{proof}
Let $v$ be a vertex of degree at least four. By the very nature of Rule $i$, it follows that $ch_i(v)\ge 0$. If $v$ discharges for Rule $i$, then $v$ never discharges again. So we may assume that $v$ only discharges during Rule $0$ and Rule $T$. For Rule $T$, $v$ sends at most $.2$ to each neighbor and only if they did not receive charge in any earlier Rule, in particular for Rule $0$. Yet in Rule $0$, $v$ sends $.4$ to each neighbor of degree three in a common triangle. 

By Lemma~\ref{NoDiamond}, these triangles are edge-disjoint. By Lemma~\ref{NoTriangle}, there is at most one vertex of degree three in each triangle. So for every neighbor of $v$ sent $.4$ charge under Rule $0$ there is a corresponding neighbor of degree at least four. Hence $v$ sends at most $.2deg(v)$ charge under Rules $0$ and $T$. As $ch(v)=deg(v)-3.2$, we find that $ch_T(v) \ge ch(v)-.2deg(v) = .8deg(v)-3.2$ which is at least $0$ since $deg(v)\ge 4$.
\end{proof}

\begin{claim}\label{CH3}
Let $C\in \C_(D_3(G_0))$ and $v\in V(C)$. If $v$ has $m$ neighbors of degree at least four that are good for $v$ and $ch_T(v) < .6$, then $ch_T(v) \ge .2(2+m-{\rm deg}_C(v))$.
\end{claim}
\begin{proof}
Note that $ch(v)=-.2$. If $v$ does not have $.6$ charge by Step $T$, then $v$ will recieve $.2$ charge from each neighor of degree at least four under Rule $T$. Moreover, $v$ recieves an additional $.2$ from each such neighbor that is good for $v$. Thus $ch_T(v)\ge .2(3-{\rm deg}_C(v)) + .2m -.2 = .2(2+m-{\rm deg}_C(v))$ as desired.



\end{proof}

Thus if all neighbors of degree at least four of vertices in $C$ are good (or adjacent to the other neighbor of degree at least four when the vertex is a leaf of $C$), Claim~\ref{CH3} says the following: If $v$ is a leaf of $C$, then $ch_T(v)\ge .6$. If $v$ has degree two in $C$, $ch_T(v)\ge .2$. Finally if $v$ has degree three in $C$, then $ch_T(v)\ge -.2$.


We need the following proposition to show that components of degree three vertices receive enough charge:

\begin{proposition}\label{IndTree}
If $H$ is a tree with maximum degree three, then $\alpha(H)\le \frac{2}{3}|V(H)|+\frac{1}{3}$.
\end{proposition}
\begin{proof}
We proceed by induction on $|V(H)|$. It is straightforward to check that the lemma holds when $|V(H)|\le 3$. So suppose $|V(H)|\ge 4$. Suppose there exists a leaf $v$ of $H$ adjacent to a vertex $u$ of degree two in $H$. Apply induction to $H'=H\setminus \{u,v\}$. Thus $\alpha(H')\le \frac{2}{3}|V(H')|+\frac{1}{3}$. Yet $|V(H)|=|V(H')|+2$ and $|\alpha(H)|=|\alpha(H')|+1$. Hence $\alpha(H)-1\le \frac{2}{3}(|V(H)|-2)+\frac{1}{3}$ and the proposition follows.

So no such pair of vertices exist. But then there must exist a vertex $v$ of degree three in $C$ that has two leaves $u_1,u_2$ as neighbors (say by considering the lowest non-leaf vertex in a depth-first search tree). Apply induction to $H'=H\setminus \{u_1,u_2,v\}$. Thus $\alpha(H')\le \frac{2}{3}|V(H')|+\frac{1}{3}$. Yet $|V(H)|=|V(H')|+3$ and $|\alpha(H)|=|\alpha(H')|+2$. Hence $\alpha(H)-2\le \frac{2}{3}(|V(H)|-3)+\frac{1}{3}$ and the proposition follows.
\end{proof}

\begin{claim}\label{Comp}
If $H$ is a component of $D_3(G_0)$, then $\sum_{v\in V(H)} ch_T(v)\ge .4\alpha(H)$.
\end{claim}
\begin{proof}
First suppose that $|V(H)|\le 3$. If $H$ is a vertex, then $ch_T(v)\ge .4 = .4\alpha(H)$ by Claim~\ref{CH3} as desired. If $H$ is an edge, then $ch_T(v)\ge .2$ by Claim~\ref{CH3} for each $v\in V(H)$. Hence $\sum_{v\in V(H)} ch_T(v)\ge .4 = .4\alpha(H)$. Hence $H$ is a path on three vertices $v_1v_2v_3$. By Lemma~\ref{Deg1}, the neighbors of degree at least four of $v_1$ are either good for $v$ or adjacent to each other. By Claim~\ref{CH3} or Rule 0, $ch_T(v_1)\ge .6$. Similarly, $ch_T(v_3)\ge .6$. Thus $\sum_{v\in V(H)} ch_T(v)\ge .6 -.2 + .6 = 1 \ge .8 = .4\alpha(H)$.

So we may assume that $|V(H)|\ge 4$. By Lemma~\ref{Components}. for every $v\in V(H)$, every neighbor of degree at least four of $v$ is either good for $v$ adjacent to another neighbor of degree at least four of $v$. This implies when combined with Rule 0 and Claim~\ref{CH3} that $ch_T(v)\ge 1-.4d_H(v)$ since leaves would have a final charge of $.6$, vertices of degree two would have a final charge of $.2$ and vertices of degree three would have a final charge of $-.2$. Thus,

$$\sum_{v\in V(H)} ch_T(v)\ge \sum_{v\in V(H)} (1-.4d_H(v)) = |V(H)| - .8 |E(H)|$$

Since $H$ is a tree, $|E(H)|=|V(H)|-1$. Hence

$$\sum_{v\in V(H)} ch_T(v)\ge \sum_{v\in V(H)} (1-.4d_H(v)) = |V(H)| - .8 (|V(H)|-1) = .2|V(H)|+.8 = .4 \frac{|V(H)|+4}{2}$$

However, $|V(H)|\le 10$ by Lemma~\ref{Components}. We claim then that for $|V(H)|\le 10$, $\frac{|V(H)|+4}{2} \ge \alpha(H)$. Suppose not. Yet $\alpha(H)\le \frac{2|V(H)|+1}{3}$ by Proposition~\ref{IndTree}. Thus $\frac{|V(H)|+4}{2} < \alpha(H) \le \frac{2|V(H)|+1}{3}$. Hence $3|V(H)|+12 < 4|V(H)|+2$. That is, $|V(H)| > 10$, a contradiction. This proves the claim. Hence 

$$\sum_{v\in V(H)} ch_T(v)\ge .4 \frac{|V(H)|+4}{2} \ge .4\alpha(H)$$

as desired.
\end{proof}

\begin{claim}
$\sum_{v\in V(G_0)} ch(v) \ge  .4\alpha(D_3(G_0))$.
\end{claim}
\begin{proof}
Note that $\sum_{v\in V(G_0)} ch(v) = \sum_{v\in V(G_0)} ch_T(v)$. By Claim~\ref{CH4}, $ch_T(v)\ge 0$ for all vertices $v$ with degree at least four. By Claim~\ref{CH3}, $\sum_{v\in V(H)} ch_T(v) \ge .4\alpha(H)$ for all components $H$ of $D_3(G_0)$. Thus $\sum_{v\in V(D_3(G))} ch_T(v) \ge .4\alpha(D_3(G_0))$ by adding over the components of $D_3(G_0)$. So $\sum_{v\in V(G_0)} ch_T(v) \ge  .4\alpha(D_3(G_0))$ as desired.
\end{proof}

Hence $|E(G_0)|\ge 1.6|V(G_0)|+.2\alpha(D_3(G_0))$ and so $p(G_0)\le 0$, a contradiction.

\end{proof}

\section{$4$-critical graphs with Ore-degree at most seven are $4$-Ore}



If $H$ is a graph, let $s(H) = |E(H)|-|V(H)|+\alpha(H)$. Note that $s(H) = \sum_{C\in \C(H)} s(H)$ and that $s(H)$ is an integer. Moreover if $H$ is connected, then $|E(H)|\ge |V(H)|-1$ and hence $s(H)\ge 0$. Thus $s(H)\ge 0$ for all graphs. We may characterize the graphs with small $s$ as follows:

\begin{proposition}\label{S}
If $H$ is connected, then
\begin{itemize}
\item $s(H)=0$ if and only if $H$ is a vertex or edge,
\item $s(H)=1$ if and only if $H$ is a triangle or a path of length two or three,
\item $s(H)=2$ if and only if $H$ is 
\begin{itemize}
\item a cycle of length four or five,
\item a triangle with a pendant edge or path of length two,
\item a tree with $\alpha(H)=3$ (i.e. a path of length four or five or a claw with up two pendant edges)
\end{itemize}
\end{itemize}
\end{proposition}
\begin{proof}
Suppose $s(H)=0$. Then $|E(H)|=|V(H)|-1$ and $\alpha(H)=1$. Thus $H$ is a tree that is also a clique, which is to say a vertex or edge.

So suppose $s(H)=1$. Then either $|E(H)|=|V(H)|-1$ and $\alpha(H)=2$, or, $|E(H)|=|V(H)|$ and $\alpha(H)=1$. In the former case, $H$ is a tree with independence number two. Thus $H$ is a path on three or four vertices. In the latter case, $H$ is a tree plus one edge but also a clique. Hence $H$ is a triangle.

Finally suppose $s(H)=2$. If $|E(H)|=|V(H)|-1$, then $H$ is a tree. In that case, $\alpha(H)=3$ as desired. If $|E(H)|=|V(H)|$, then $H$ is a tree plus an edge and also $\alpha(H)=2$. Now $H$ contains a cycle $C$. $C$ has length at most five. If $|C|=4$ or $5$, then it follows that $H=C$ as desired. So we may suppose that $|C|=3$. But then it follows once again by independence number that $H$ is a triangle with at most one pendant tree and in fact that tree is a path on one or two vertices as desired.
\end{proof}

Here is another useful lemma:

\begin{lemma}\label{sOreDeg7}
If $G$ is a $4$-critical graph of Ore-degree at most seven, then $p(G)=.6 s(D_3(G))$.
\end{lemma}
\begin{proof}
Let $m$ be the number of vertices of degree 4. Thus $|E(G)|= 4m + |E(D_3(G))|$ and $|V(G)|=m+|V(D_3(G))|$. Hence 

$$p(G) = 4.8(m+|V(D_3(G))|) -3(4m +|E(D_3(G))|) + .6\alpha(D_3(G)) = -7.2m + p(D_3(G))$$

Furthermore, we may note that $p(D_3(G)) = \sum_{C\in \C(D_3(G))} p(C)$, where $\C(D_3(G))$ denotes the connected components of $D_3(G)$. Hence,

$$p(G) = \sum_{C\in \C(D_3(G))} (p(C) - 1.8|E(V(C), V(G\setminus C))|)$$

However, as all of the vertices in $D_3(G)$ have degree three in $G$ it is easy to note that $|E(V(C), V(G\setminus C)| = 3|V(C)|-2|E(C)|$ by summing degrees. Thus

$$p(G) = \sum_{C\in \C(D_3(G))} .6(|E(C)|-|V(C)|+\alpha(C)) = .6s(D_3(G))$$.
\end{proof}

\begin{theorem}
If $G$ is a $4$-critical graph with Ore-degree at most seven, then $G$ is $4$-Ore.
\end{theorem}
\begin{proof}
Let $G$ be a counterexample with a minimum number of vertices. By Theorem~\ref{Ore4Best}, $p(G)\le 1.2$. By Lemma~\ref{sOreDeg7}, $s(D_3(G))\le 2$. That is, $\sum_{C\in \C(D_3(G))} s(C) \le 2$. As $s(D_3(G))\ge 0$, we find that $p(G)\ge 0$. Yet, as $G$ is not $3$-colorable, $D_3(G)$ is not bipartite. Hence there must be at least one component $C_1$ of $D_3(G)$ with an odd cycle. It follows that $s(C_1)\ge 1$. Thus $p(G)\ge .6$. 

\begin{claim}\label{No2Cut}
There does not exist a $2$-vertex-separation of $G$.
\end{claim}
\begin{proof}
Suppose not. But then $G$ is the Ore-composition of two graphs $G_1$ and $G_2$ by Proposition~\ref{OreSep}. Yet as $G$ has Ore-degree at most seven, it follows from Lemma~\ref{4CritOre7} that $G_1$ and $G_2$ have Ore-degree at most seven. By the minimality of $G$, $G_1$ and $G_2$ are $4$-Ore. Thus $G$ is also $4$-Ore, a contradiction.
\end{proof}

\begin{claim}\label{NoTriangle2}
There does not exist a triangle in $G$ with exactly two vertices of degree three.
\end{claim}
\begin{proof}
Let $T=u_1u_2v$ be the triangle with $d(u_1)=d(u_2)=3$ and $d(v)=4$. Let $z_1$ be the other neighbor of $u_1$ and $z_2$ the other neighbor of $u_2$.

If $z_1=z_2$, then $\{z_1,v\}$ is a $2$-separation of $G$ contradicting Claim~\ref{No2Cut}. So we may suppose that $z_1\ne z_2$. Consider $G' = G\setminus\{u_1,u_2\} + z_1z_2$. As $G$ is not $3$-colorable, neither is $G'$. So $G'$ contains a $4$-critical subgraph $K$. Let $R=V(K)$. Note that $R\cap T=\emptyset$.



Let $S$ be the boundary of $R$. As every vertex of $K$ has degree at least three in $K$, it follows that all the vertices in $S\setminus \{z_1,z_2\}$ have degree four in $G$. First suppose that at least one of $z_1,z_2$ is degree three in $K$. In this case, then $K$ has Ore-degree at most seven. To see this, note that every edge $e\in E(K)\cap E(G)$ has Ore-degree at most that it had in $G$ which is seven.  The only other edge is $z_1z_2$ but since at least one of $z_1,z_2$ has degree three in $K$, the Ore-degree of $K$ is at most seven. 

By Lemma~\ref{sOreDeg7}, $p(K)= .6s(D_3(K)) = .6\sum_{C\in \C(D_3(K))} s(C)$.  Let $\C_K$ be the components of $D_3(G)\cap K$. As every component in $\C_K$ is an induced subgraph of a component in $D_3(G)$, it follows that $\sum_{C\in \C(D_3(K))}\alpha(C) \ge \sum_{C\in \C_K} \alpha(C)$. 

Let $S'=D_3(K)\setminus D_3(G)$. Hence $S\supseteq S'\supseteq S\setminus \{z_1,z_2\}$. As every vertex in $S'$ has neighbors of only degree three in $G$, it follows that $\sum_{C\in \C(D_3(K))}|E(C)| = 3(|S'|)  + \sum_{C\in \C_K} |E(C)|$ while $\sum_{C\in \C(D_3(K))}|V(C)| = |S'| + \sum_{C\in \C_K} |V(C)|$. Thus $p(K) \ge \sum_{C\in C_K}s(C) + 1.2|S'|$. Yet $s(C)\ge 0$ for all $C\in C_K$. Thus $p(K)\ge 1.2 |S'|$. This implies that $|S'|=1$. But then $G$ has a $2$-separation formed by $v$ and the unique element of $S'$, contradicting Claim~\ref{No2Cut}.

So we may assume that $z_1$ and $z_2$ have degree four in $K$. Similar calculations ($p(K)=.6s(D_3(K))-.6$) show that $|S'|\le 2$ in this case. Hence $|S|\le 4$. Let $R'$ be a critical extension of $R$ with extender $W$ such that if possible the extension is not total, and subject to that has maximum core size. Since $|S|\le 4$, it is not hard to see that $W$ has Ore-degree at most seven. Hence $W$ is $4$-Ore as $G$ is a minimum counterexample. 

If $W$ is not total or has core size two, then $p(G)\le p(K) + p(W) - 3.6$ as every vertex in $S$ has degree at least four in $G$. As $G$ is a minimum counterexample, $p(K),p(W)\le 1.8$ and so $p(G)\le 0$, a contradiction. Thus every critical extension of $R$ is total.
 
Next suppose that $W$ has a core of size three. Let the three vertices identified in $W$ be $w_1, w_2, w_3$. Since $z_1,z_2$ must be identified to one vertex, say $w_1$, we find that $|S|=4$ and the other two vertices in $S$ must be identified to $w_2$ and $w_3$ respectively. But now after collapsing the diamond $w_1u_1u_2v$ in $W$, there must be a separation of order two. This implies that there exists a separation of order two in $G$, contradicting Claim~\ref{No2Cut}.

So we may assume that every critical extension of $R$ has a core of size one and is total. By Lemma~\ref{Coll}, it follows that $R$ is collapsible. But then $R\cup T$ is collapsible. Yet the boundary of $R\cup T$ contains a vertex with degree two in $R\cup T$, namely $z$, a contradiction as in the above calculation.
\end{proof}

\begin{claim}
There exist at least two components of $D_3(G)$ containing odd cycles.
\end{claim}
\begin{proof}
Suppose that $C_1$ is the only component of $D_3(G)$ containing an odd cycle. Let $D_4(G)$ denote the vertices of degree four in $G$.

First suppose $s(C_1)=2$. Since $D_3(G)$ cannot be bipartite, it follows that $C_1$ is a cycle of length five. Let $v\in V(C_1)$ and $w$ be the neighbor of $v$ of degree at least four. Note that $D_3(G)\cup\{w\}\setminus\{v\}$ is bipartite by Claim~\ref{NoTriangle2}. Now color $D_4(G)\cup\{v\}\setminus\{w\}$ with color $3$ and extend this coloring to $G$, a contradiction. 

So we may assume that $s(C_1)=1$ and hence $C_1$ is a triangle by Proposition~\ref{S}. First suppose that $s(D_3(G))=1$. Let $v$ be a neighbor of degree four of a vertex $w$ in $C_1$. Note that $D_3(G)\cup \{v\} \setminus \{w\}$ is bipartite since $v$ is not adjacent to two adjacent vertices of degree three by Claim~\ref{NoTriangle2}. Now we may color $D_4(G)\cup\{w\} \setminus \{v\}$ with color $3$ and extend the coloring to the rest of $G$ which is bipartite, a contradiction. 

So we may suppose that $s(D_3(G))=2$. That is, there exists exactly one component $C_2\ne C_1$ of $D_3(G)$ such that $s(C_2)=1$. As we supposed that $C_1$ is the only component of $G$ containing an odd cycle, $C_2$ is a path $P=p_1\ldots p_k$ where $k=3$ or $4$ by Proposition~\ref{S}. Let $v$ be a neighbor of degree at least four of a vertex in $C_1$ such that $v$ is not adjacent to both $p_1$ and $p_k$. Such a vertex $v$ exists since the three neighbors of vertices in $C_1$ with degree at least four are distinct by Claim~\ref{NoTriangle2}, and there can be at most two vertices of degree at least four adjacent to both $p_1$ and $p_k$. Let $w$ denote the neighbor of $v$ in $C_1$. Note that $D_3(G)\cup \{v\} \setminus \{w\}$ is bipartite since $v$ is not adjacent to two adjacent vertices of degree three by Claim~\ref{NoTriangle2} and $v$ is not in an odd cycle with the vertices of $C_2$ given the choice of $v$. So we may color $D_4(G)\cup\{w\} \setminus \{v\}$ with color $3$ and extend the coloring to the rest of $G$ which is bipartite, a contradiction. 
\end{proof}

So we may assume there exist at least two components $C_1,C_2$ containing an odd cycle. Yet $s(C_1)+s(C_2)\le 2$ as $s(D_3(G))\le 2$. So $s(C_1)=s(C_2)=1$ and thus $C_1,C_2$ are triangles by Proposition~\ref{S}. Let $V(C_1)=\{a_1,a_2,a_3\}$ and $V(C_2)=\{b_1,b_2,b_3\}$. Let $u_i$ denote the neighbor of degree at least four of $a_i$ and $v_j$ denote the neighbors of degree at least four of$b_j$. By Claim~\ref{NoTriangle2}, $u_i\ne u_j$ and $v_i\ne v_j$ for all $i,j \in \{1,2,3\}$.

\begin{claim}\label{UVdistinct}
All of $u_1,u_2,u_3,v_1,v_2,v_3$ are distinct.
\end{claim}
\begin{proof}
Suppose not. Thus there exist $i,j$ such that $u_i=v_j$.  Note that $D_3(G)\cup \{v\} \setminus \{a_i,b_j\}$ is bipartite since $v$ is not adjacent to two adjacent vertices of degree three by Claim~\ref{NoTriangle2}. So we may color $D_4(G)\cup\{a_i,b_j\} \setminus \{v\}$ with color $3$ and extend the coloring to the rest of $G$ which is bipartite, a contradiction. 
\end{proof}

\begin{claim}\label{5Cycle}
For all $i,j\in \{1,2,3\}$, there exists a $5$-cycle $C_{ij}$ containing $u_i$ and $v_j$ and contained in $D_3(G)\cup\{u_i,v_j\}-(C_1\cup C_2)$.
\end{claim}
\begin{proof}
Let $i,j\in\{1,2,3\}$. Let $\phi_{ij}$ be the coloring that assigns the color 3 to $D_4(G)\cup\{a_i,b_j\}\setminus \{u_i,v_j\}$. As $G$ is not $3$-colorable, $\phi_{ij}$ cannot be extend to a $3$-coloring of $G$ for all $i$ and $j$. That is, $D_3(G)\cup\{u_i,v_j\}-\{a_i,b_j\}$ is not bipartite and so contains an odd cycle $C_{ij}$. But then $C_{ij}$ must have size five and contain $u_i$ and $v_j$. Note that $C_{ij}$ does not contain any vertex in $C_1\cup C_2$ since these vertices are either $a_i$ or $b_j$ or are in a different component of $D_3(G)\cup\{u_i,v_j\}-\{a_i,b_j\}$ than $u_i$ and $v_j$.
\end{proof}

For each $i,j\in\{1,2,3\}$, let $C_{ij}=u_ix_{ij}v_jz_{ij}y_{ij}$. That is to say that $u_i$ and $v_j$ have a neighbor of degree three $x_{ij}$ in common as well as a path of length three $u_iy_{ij}z_{ij}v_j$ whose internal vertices are degree three. Let $X=\{x_{ij}|1\le i,j\le 3\}$, $Y=\{y_{ij}| 1\le i,j \le 3\}$ and $Z=\{z_{ij}| 1\le i,j\le 3\}$. Note that $(X\cup Y \cup Z)\cap (C_1\cap C_2) = \emptyset$ since each $C_{ij}$ is disjoint from $C_1\cup C_2$.

\begin{claim}\label{Xdistinct}
All $x_{ij}$ are distinct.
\end{claim}
\begin{proof}
Suppose not. Note that $x_{ij}\ne x_{kl}$ when $i\ne k$ and $j\ne l$ since otherwise that vertex in $X$ has degree four since its neighbors $u_i,u_k,v_j,v_l$ are all distinct, a contradiction.

We may suppose without loss of generality then, by the symmetry of $C_1$ and $C_2$, that there exist $i,j,k$ with $j\ne k$ such that $x_{ij}=x_{ik}$. Suppose without loss of generality that $i=j=1,k=2$. But then $x_{11}=x_{12}$ is a vertex of degree three with neighbors $u_1,v_1,v_2$. That says that all the neighbors of $x_{11}$ are degree four. Hence $x_{11}\not\in Y\cup Z$ since every vertex of $Y\cup Z$ has a neighbor of degree three. Let $w_1,w_2$ be the neighbors $w_1,w_2$ of $u_1$ distinct from $a_1$ and $x_{11}$. Since the neighbors of $u_1$ corresponds to the three elements $y_{1m}$ of $Y$, it follows that one of these neighbors corresponds to at least two such elements of $Y$. 

We may suppose without loss of generality that $w_1$ is that vertex. Hence $w_1$ has a neighbor of degree three, call it $v$ such that $t$ has two neighbors in $\{v_1,v_2,v_3\}$. Hence at least one of $v_1,v_2$ is a neighbor of $t$. We may suppose without loss of generality that $v_1$ is a neighbor of $t$. But then $v_1$ has neighbors $b_1, x_{11}, t$ and a fourth neighbor, call it $s$. Now $v_1$ has a common neighbor with both $u_2$ and $u_3$ distinct from $b_1$. But $x_{11}$ and $t$ are not incident with $u_2$ or $u_3$. Hence $s$ is adjacent to both $u_2$ and $u_3$. But then $v_1$ is not in path of length three $u_ny_{n1}z_{n1}v_1$ with at least one $n\in \{2,3\}$. This follows since such a path would have to go through $t$ and $w_1$. But $w_1$ has neighbors $t$ and $u_1$ and hence is adjacent to at most one of $u_2$ or $u_3$. Thus the pair $u_n,v_1$ is not in a cycle $C_{n1}$, contradicting Claim~\ref{5Cycle}. 
\end{proof}

Since each $u_i$ has at most three neighbors outside $V(C_1)\cup V(C_2)$, it follows that $\bigcup_i N_(u_i) \setminus V(C_1) = X$. Similarly, $\bigcup_j N(v_j)\setminus V(C_2) = X$. 

\begin{claim}\label{XZpaired}
For each $i,j\in\{1,2,3\}$, $x_{ij}=z_{kl}$ for at most one pair $k,l$ and only then if $l=j$ and $x_{ij}$ is adjacent to $x_{km}$ for some $m$.
\end{claim}
\begin{proof}
This follow since $x_{ij}$ is adjacent to $u_i$ and $v_j$ and a vertex of degree three call it $w$. But then since $x_{ij}=z_{kl}$, $u_kwz_{kl}v_l$ is a path. Since $v_l\ne u_i$ by Claim~\ref{UVdistinct}, we find that $v_l=v_j$, i.e. $l=j$. Furthermore, by Claim~\ref{Xdistinct}, it follows that $w\in X$  and hence $w=x_{km}$ for some $m$. But then $k,l$ can be the only such pair as claimed.
\end{proof}

Now every element of $Z$ corresponds to at most one $X$ since the vertices in $X$ are distinct by Claim~\ref{Xdistinct}. Hence by Claim~\ref{XZpaired} each $z_{kl}$ is a distinct vertex. Thus $|Z|=9$. But then since $|X|=|Z|=9$, every element of $Z$ correspends to exactly one $X$. On the other hand, every element of $X\cap Z$ is adjacent to another element of $X$ by Claim~\ref{XZpaired}. However, every element of $X$ has at most one neighbor of degree three and hence at most neighbor in $X$. This implies then that $X$ induces a matching. This is a contradiction since the size of $X$ is odd. 
\end{proof}

\section{Open Questions}

Kostochka and Yancey characterized the $4$-critical graphs that satisfy $|E(G)|=\frac{5|V(G)|-2}{3}$ as precisely the $4$-Ore graphs. A much shorter proof of Theorem~\ref{Main} would exist if the next \'level' of graphs could be characterized:

\begin{question}
What is the structure of the $4$-critical graphs that satisfy $|E(G)|=\frac{5|V(G)|-1}{3}$?
\end{question}

However, this seems like a difficult problem given the examples provided by Kostochka and Yancey. While some description may be possible, it seems hard to find some characterization that would imply that such graphs contain two adjacent vertices of degree at least four as Theorem~\ref{Main} would require. Nevertheless, such a characterization would be of interest.

On the other hand, Theorem~\ref{Ore4Best} seems interesting in its own right as it incorporates independence number into the bound. In particular, Theorem~\ref{Ore4Best} improves Kostochka and Yancey's bound from Theorem~\ref{Ore4} when $\alpha(D_3(G))\ge |V(G)|/3$. Of course, $4$-Ore graphs satisfy $\alpha(D_3(G))\le \frac{|V(G)|-1}{3}$. This raises the question whether such a result could be true for general $k$. Note that $k$-Ore graphs satisfy $\alpha(D_{k-1}(G))\le \frac{|V(G)|-1}{k-1}$. What happens then for larger values of $\alpha(D_{k-1}(G))$? Could there be an improvement in the edge density for such graphs?  Well, a result of Kierstead and Rabern~\cite{KR} implies an affimative answer and even a linear improvement when $\alpha(G) \ge \frac{2(k-2)}{(k-1)^2} |V(G)|$. Their more general result is the following (cf. 4.5 in ~\cite{KR}):

\begin{thm}
If $G$ is a $k$-critical graph, then 

$$2|E(G)|\ge (k-2)|V(G)| + {\rm mic}(G) + 1$$

where ${\rm mic}(G) = \max\{ \sum_{v\in I} d(v) | I\subseteq V(G) {\rm~independent} \}$.
\end{thm}
  
It is natural then to wonder if an improvement exists for the range $\frac{1}{k-1} \le \frac{\alpha(D_{k-1}(G))}{|V(G)|} \le \frac{2(k-2)}{(k-1)^2}$. Hence we pose the following question:

\begin{question}
Can Kostochka and Yancey's general bound for a $k$-critical graph $G$ be improved when $\alpha(D_{k-1}(G))\ge \frac{|V(G)|}{k-1}$? In particular is there a linear improvement as in Theorem~\ref{Ore4Best}?
\end{question}

\end{document}